\documentclass[a4paper,12pt]{article}

\usepackage[utf8]{inputenc}
\usepackage[T1]{fontenc}
\usepackage{lmodern}
\usepackage{graphicx}  
\usepackage{amsmath} 
\usepackage{amssymb}  
\usepackage{color}

\usepackage{amsthm}
\usepackage{hyperref}
\usepackage{subfigure}
\usepackage{enumerate}
\usepackage{caption}
\usepackage{multirow}

\def\sizesmallfig{0.30}
\def\sizetinyfig{0.24}

\newcommand{\R}{\mathbb{R}}

\newcommand{\N}{\mathbb{N}}
\newcommand{\x}{\mathbf{x}}
\newcommand{\y}{\mathbf{y}}
\newcommand{\z}{\mathbf{z}}

\newcommand{\f}{\mathbf{f}}

\def\Q{\mathbf{Q}}
\def\q{\mathbf{q}}
\newcommand{\M}{\mathbf{M}}

\def\f{f}

\def\p{\mathbf{p}}
\def\S{\mathbf{S}}
\def\B{\mathbf{B}}
\def\K{\mathbf{K}}
\def\F{\mathbf{F}}
\newcommand{\A}{\mathbf{A}}

\newcommand{\nlift}{n_{\text{lift}}}
\newcommand{\Slift}{\S_{\text{lift}}}

\newcommand{\dlift}{r_{\text{lift}}}

\DeclareMathOperator{\vol}{vol}

\theoremstyle{plain}
\newtheorem{theorem}{Theorem}[section]
\newtheorem{lemma}[theorem]{Lemma}
\newtheorem{proposition}[theorem]{Proposition}
\newtheorem{corollary}[theorem]{Corollary}

\theoremstyle{definition}
\newtheorem{definition}[theorem]{Definition}

\newtheorem{assumption}[theorem]{Assumption}
\newtheorem{example}{Example}

\theoremstyle{remark}

\title{\bf Semidefinite approximations of projections and polynomial images of semi-algebraic sets}
\begin{document}
\author{ Victor Magron$^{1}$ \and Didier Henrion$^{2,3,4}$ \and Jean-Bernard Lasserre$^{2,3}$}

\footnotetext[1]{Circuits and Systems Group, Department of Electrical and Electronic Engineering,
Imperial College London, South Kensington Campus, London SW7 2AZ, UK.}
\footnotetext[2]{CNRS; LAAS; 7 avenue du colonel Roche, F-31400 Toulouse; France.}
\footnotetext[3]{Universit\'e de Toulouse;  LAAS, F-31400 Toulouse, France.}
\footnotetext[4]{Faculty of Electrical Engineering, Czech Technical University in Prague,
Technick\'a 2, CZ-16626 Prague, Czech Republic}

\date{\today}

\maketitle

\begin{abstract}
Given a compact semi-algebraic set $\S \subset \R^n$ and a polynomial map $\f : \R^n \to \R^m$, we consider the problem of approximating the image set $\F=\f(\S) \subset \R^m$. This includes in particular the projection of $\S$ on $\R^m$ for $n \geq m$.
Assuming that $\F \subset \B$, with $\B \subset \R^m$ being a ``simple'' set (e.g. a box or a ball), we provide two methods to compute certified outer approximations of $\F$. 
Method 1 exploits the fact that  $\F$ can be defined with an existential quantifier, while 
Method 2 computes approximations of the support of image measures. 
The two methods output a sequence of superlevel sets defined with a single polynomial that yield explicit outer approximations of $\F$. Finding the coefficients of this polynomial boils down to computing an optimal solution of a convex semidefinite program. We provide guarantees of strong convergence to $\F$ in $L_1$ norm on $\B$, when the degree of the polynomial approximation tends to infinity. Several examples of applications are provided, together with numerical experiments.
\end{abstract}
\section{Introduction}
\label{sec:intro}
Consider a polynomial map $\f : \R^n \to \R^m$, $\x \mapsto \f(\x) := (f_1(\x), \dots, f_m(\x)) \in \R^m[\x]$ 
of degree $d:=\max \{\deg f_1,\ldots,\deg f_m\}$ and a compact
basic semi-algebraic set
\begin{equation}
\label{eq:defS}
 \S :=  \{\x \in \R^n : g_1^\S(\x)  \geq 0, \dots, g_{n^\S}^\S(\x) \geq 0 \}
\end{equation} 
defined by polynomials $g_1^\S,\ldots,g_{n^\S}^\S \in \R[\x]$.

Since $\S$ is compact, the image set
\[
\F:=f(\S)
\]
is included in a basic compact semi-algebraic set $\B$, assumed to be ``simple'' (e.g. a box or a ball)
and described by
\begin{equation}
\label{eq:defB} 
\B :=  \{\y \in \R^m : g_1^\B(\y)  \geq 0, \dots, g_{n^\B}^\B(\y) \geq 0 \}
 \end{equation} 
for some polynomials $g_1^\B,\ldots,g_{n^\B}^\B \in \R[\y]$.

The purpose of this paper is to approximate $\F$, the image of $\S$ under the polynomial map $f$, with superlevel sets of single polynomials of fixed degrees. One expects the approximation to be {\em tractable}, i.e. to be able to control the degree of the polynomials used to define the approximations. This appears to be quite a challenging problem since the polynomial map $f$ and the set $\S$ can be both complicated. 

This problem includes two important special cases. The first problem is to approximate the projection of $\S$ on $\R^m$ for $n \geq m$. The second problem is the approximation of Pareto curves in the context of multicriteria optimization. In~\cite{orl14}, we reformulate this second problem through parametric polynomial optimization, which can be solved using a hierarchy of semidefinite approximations. The present work proposes an alternative solution via approximations of polynomial images of semi-algebraic sets.

In the case of semi-algebraic set projections, notice that computer algebra algorithms provide an exact description of the projection.
 %
These algorithms are based on real quantifier elimination (see e.g.~\cite{Tarski51,Collins74,Basu96}). 
For state-of-the-art computer algebra algorithms for quantifier elimination, we refer the interested reader to the survey~\cite{Basu14survey} and the references therein. Quantifier elimination can be performed with the famous cylindrical algebraic decomposition algorithm. 
For a finite set of $s$ polynomials in $n$ variables, the (time) computational complexity of this algorithm is bounded by $(sd)^{2^{O(n)}}$, thus doubly exponential~\cite{Collins74,Wuthrich76}.
In~\cite{Grigorev1988} an algorithm was proposed to find real elements of semi-algebraic sets in sub-exponential time. The {\em Block Elimination Algorithm} is a singly exponential algorithm to eliminate one block of $n$ variables out of $n+m$ variables, with a  complexity bounded by $s^{n+1} d^{\, O(n+m)}$ (see~\cite[Chapter 14]{Basu06} for the formalization of this algorithm).
For applications that satisfy certain additional assumptions (e.g. radicality, equidimensionality, etc.), one can use the variant quantifier elimination method proposed in~\cite{journals/jsc/HongD12}, which is less computationally demanding.

Providing approximation algorithms for quantifier elimination is interesting on its own because it may provide simpler answers than exact methods, with a more reasonable computational cost.
On the one hand, we do not require an exact description
of the projection but rather a hierarchy of outer approximations with
a guarantee of convergence to the exact projection. On the other hand, the present methodology only requires the following assumptions: 1) the set $\S$ is compact and 2) either the semi-algebraic set $\S$ or $\F$ (resp. $\B \backslash \F$) has nonempty interior. 
%
\paragraph*{Contribution and general methodology}
We provide two methods to approximate the image of semi-algebraic sets under polynomial applications. 
\begin{itemize}
\item Method 1 consists of rewriting $\F$ as a set defined with an existential quantifier.  Then, one can outer approximate $\F$ as closely as desired with a hierarchy of superlevel sets of the form $\F^1_r := \{ \y \in \B : q_r (\y) \geq 0 \}$ for some polynomials $q_r \in \R[\y]$ of increasing degrees $2r$.
\item Method 2 consists of building a hierarchy of relaxations for the infinite dimensional moment problem whose optimal value is the volume of $\F$ and whose optimum is the restriction of the Lebesgue measure on $\F$. Then, one can outer approximate $\F$ as closely as desired with a hierarchy of super level sets of the form $\F^2_r := \{ \y \in \B : w_r (\y) \geq 1 \}$, for some polynomials $w_r \in \R[\y]$ of increasing degrees $2r$. 
\end{itemize}
Method 1 and Method 2 share the following essential features:
\begin{enumerate}
\item The sets $\F^1_r$ and $\F^2_r$ are described with a single polynomial of degree $2 r$.
\item Assuming non-emptiness of the interior of $\S$, resp. of $\F$ and $\B \backslash \F$, one has\\ $\lim_{r\to\infty} \vol (\F^1_r \backslash  \F) = 0$, resp. $\lim_{r\to\infty} \vol (\F^2_r \backslash  \F) = 0$, where $\vol(\cdot)$ stands for the volume or Lebesgue measure.
\item Computing the coefficient vectors of the polynomials $(q_r)_{r \in \N}$, resp. $(w_r)_{r \in \N}$, boils down to finding optimal solutions of a hierarchy of semidefinite programs. The size of these programs is parametrized by the relaxation order $r$ and depends on the number of variables $n$, the number $m$ of components of the polynomial $f$ as well as its degree $d$. For the hierarchy of semidefinite programs associated with Method 1, the number of variables at step $r$ is bounded by $\binom{n + m + 2 r}{2 r}$, with $(n^\S + n^\B + 1)$ semidefinite constraints of size at most $\binom{n + m + r}{r}$. Step $r$ of the semidefinite hierarchy associated with Method 2 involves at most $\binom{n + 2 r d}{2 r d} + 2 \binom{m + 2 r}{2 r}$ variables, $(n_\S+1)$ semidefinite constraints of size at most $\binom{n + r d}{r d}$ and $2 (n^\B+1)$ semidefinite constraints of size at most $\binom{m + r}{r}$. 
\item Data sparsity can be exploited to reduce the overall computational cost.
\end{enumerate}

Method 1 relies on the previous study~\cite{Las13}, in which the author obtains tractable approximations of sets defined with existential quantifiers. The present article provides an extension of the result of~\cite[Theorem 3.4]{Las13}, where one does not require anymore that some set has zero Lebesgue measure. This is mandatory to prove the volume convergence result.

In~\cite{HLS08vol}, the authors consider the problem of approximating the volume of a general compact basic semi-algebraic set. The initial problem is then reformulated as an infinite dimensional linear programming (LP) problem, whose unknown is the restriction of the Lebesgue measure on the set of interest. The main idea behind Method 2 is a similar infinite dimensional LP reformulation of the problem, whose unknown is $\mu_1$, the restriction of the Lebesgue measure on $\F$. One ends up in computing a finite number of moments of the measure $\mu_0$ supported on $\S$ such that the image of $\mu_0$ under $f$ is precisely $\mu_1$. Note however that there is an important novelty compared with~\cite{HLS08vol}, in which the set under study is explicitly described as a basic compact semi-algebraic set (i.e. the intersection of superlevel sets of known polynomials), whereas such a description is not known for $\F$.

%
\paragraph*{Structure of the paper}
The paper is organized as follows. Section~\ref{sec:prelim} recalls the basic background about polynomial sum of squares approximations, moment and localizing matrices. Section~\ref{sec:exists} presents our approximation method for existential quantifier elimination (Method 1). Section~\ref{sec:imagemeasure} is dedicated to the support of image measures (Method 2). In Section~\ref{sec:compl}, we analyze the theoretical complexity of both methods and describe how the system sparsity can be exploited. Section~\ref{sec:benchs} presents several examples where  Method 1 and Method 2 are successfully applied.

\section{Notation and Definitions}
\label{sec:prelim}
Let $\R[\x]$ (resp.~$\R_{2r}[\x]$) be the ring of real polynomials (resp. of degree at most $2r$) in the variable
 $\x=(x_1,\ldots,x_n) \in \R^n$, for $r \in \N$. With $\S$ a basic semi-algebraic set as in~\eqref{eq:defS}, we set $r_j^\S := \lceil (\deg g_j^\S ) / 2 \rceil, j = 1, \dots, n^\S$ and with $\B$ a basic semi-algebraic set as in~\eqref{eq:defB}, we set $r_j^\B := \lceil (\deg g_j^\B ) / 2 \rceil, j = 1, \dots, n^\B$.
Let $\Sigma[\x]$ denote the cone of sum of squares (SOS) of polynomial, and let $\Sigma_r[\x]$ denote the cone of
polynomials SOS of degree at most $2 r$, that is $\Sigma_r[\x] := \Sigma[\x] \cap \R_{2r}[\x]$.
 
For the ease of notation, we set $g_0^\S(\x) := 1$ and $g_0^\B(\y) := 1$.
For each $r \in \N$, let $\Q_r(\S)$ (resp. $\Q_r(\B)$) be the $r$-truncated quadratic module (a convex cone) generated by $g_0^\S, \dots, g_{n_\S}^\S$ (resp. $g_0^\B, \dots, g_{n_\B}^\B$):
\begin{align*}
\Q_r ({\S}) & := \Bigl\{\,\sum_{j=0}^{n^\S} s_j(\x) {g_j^\S} (\x) : s_j \in \Sigma_{r -  r_j^\S}[\x], \,  j = 0, \dots, n^\S \, \Bigr\},\\
\Q_r ({\B}) & := \Bigl\{\,\sum_{j=0}^{n^\B} s_j(\y) {g_j^\B} (\y) : s_j \in \Sigma_{r - r_j^\B}[\y], \,  j = 0, \dots, n^\B  \,\Bigr\}.
\end{align*}

Now, we introduce additional notations which are required for Method 1.
Let us first describe the product set
\begin{equation}
\label{eq:defK} 
\K := \S \times \B = \{(\x, \y) \in \R^{n+m} : g_1 (\x, \y)  \geq 0, \dots, g_{n^\K}(\x, \y) \geq 0 \}  \subset \R^{n+m},
\end{equation}
with $n^\K := n^\S + n^\B$ and the polynomials $g_j \in \R[\x,\y], j=1,\dots, n^\K$ are defined by:
\[   
g_j (\x, \y) := 
     \begin{cases}
       g_j^\S(\x) &\text{if} \  1 \leq j \leq n^\S,\\
       g_j^\B(\y) &\text{if} \ n^\S + 1 \leq j \leq n^\K.
     \end{cases}
\]
As previously, we set $r_j^\K := \lceil (\deg g_j)/2 \rceil, j = 1, \dots, n^\K$, $g_0(\x, \y) := 1$ and $\Q_r ({\K})$ stands for the $r$-truncated quadratic module generated by the polynomials $g_0, \dots, g_{n^\K}$. 

To guarantee the theoretical convergence of our two methods, we need to assume the existence of the following
algebraic certificates of boundedness of the sets $\S$ and $\B$:
\begin{assumption}
\label{hyp:archimedean}
There exists an integer $j^\S$ (resp. $j^\B$) such that $g_{j^\S}^\S = g^\S := N^\S - \| \x \|_2^2$ (resp. $g_{j^\B}^\B = g^\B := N^\B - \| \y \|_2^2$) for large enough  positive integers $N^\S$ and $N^\B$. 
\end{assumption}

For every $\alpha\in\N^n$ the notation $\x^\alpha$ stands for the monomial $x_1^{\alpha_1}\dots x_n^{\alpha_n}$ and for every $r \in \N$, let $\N^{n}_r := \{ \alpha \in \N^{n} : \sum_{j=1}^{n} \alpha_j \leq r \}$, whose cardinality is $\binom{n+r}{r}$.
One writes a polynomial $p \in \R[\x, \y]$ as follows:
\[(\x, \y) \mapsto p(\x, \y)\,=\,\sum_{(\alpha, \beta)\in\N^{n + m}}\,p_{\alpha \beta} \, \x^\alpha \y^\beta \:, \]
and we identify $p$ with its vector of coefficients $\p=(p_{\alpha \beta})$ in the canonical basis $(\x^\alpha \y^\beta)$, $\alpha\in\N^n$, $\beta\in\N^m$.

Given a real sequence $\z =(z_{\alpha \beta})$, we define the multivariate linear functional $\ell_\z : \R[\x,\y] \to \R$ by $\ell_\z(p) := \sum_{\alpha \beta} p_{\alpha \beta} z_{\alpha \beta}$, for all $p \in \R[\x, \y]$.

\paragraph*{Moment matrix}
The {\it moment} matrix associated with a sequence
$\z=(z_{\alpha \beta})_{(\alpha, \beta) \in \N^{n+m}}$, 
is the real symmetric matrix $\M_r(\z)$ with rows and columns indexed by $\N_r^{n+m}$, and whose entries are defined by: 
\[ 
\M_r(\z)((\alpha, \beta),(\delta, \gamma)) := \ell_\z(\x^{\alpha + \delta} \y^{\beta+\gamma})  , \quad
\forall \alpha, \delta \in \N_r^n ,  \quad
\forall \beta, \gamma \in \N_r^m . 
\]
\paragraph*{Localizing matrix}
The {\it localizing} matrix associated with a sequence
$\z=(z_{\alpha \beta})_{(\alpha, \beta) \in \N^{n+m}}$ and a polynomial $q\in\R[\x, \y]$ (with $q(\x,\y)=\sum_{u,v} q_{u v} \x^u \y^v$)
is the real symmetric matrix $\M_r(q \z)$ with rows and columns indexed by $\N_r^{n+m}$, and whose entries are defined by: 
\[ 
\M_r(q \z) ((\alpha, \beta),(\delta, \gamma)) := \ell_\z(q(\x,\y) \x^{\alpha + \delta} \y^{\beta+\gamma})  , \quad
\forall \alpha, \delta \in \N_r^n ,  \quad
\forall \beta, \gamma \in \N_r^m . 
\]

We define the restriction of the Lebesgue measure on a subset $\A \subset \B$
by $\lambda_\A (d \y) := \mathbf{1}_\A(\y) \, d \y $, 
with $\mathbf{1}_\A : \B \to \{0, 1\}$ denoting the indicator function on $\A$:
\[   
\mathbf{1}_\A (\y) := 
     \begin{cases}
       1 &\text{if } \y \in \A,\\
       0 &\text{otherwise}.
     \end{cases}
\]
The moments of the Lebesgue measure on $\B$ are denoted by
\begin{equation}\label{momb}
z^\B_{\beta} := \int \y^{\beta} \lambda_\B(d\y) \in \R, \quad \beta \in \N^m
\end{equation}
We assume that the bounding set $\B$ is ``simple'' in the following sense:
\begin{assumption} 
The moments~\eqref{momb} of the Lebesgue measure on $\B$ can be explicitly computed using cubature formula for integration.
\end{assumption}

\section{Method 1: existential quantifier elimination}
\label{sec:exists}

\subsection{Semi-algebraic sets defined with existential quantifiers}
\label{subsec:exists}

The set $\F=\f(\S)$ is the image of the compact semi-algebraic set $\S$ under the polynomial map $\f : \S \to \B $, thus
it can be defined with an existential quantifier:
\[
\F = \{ \y \in {\B} : \exists\: \x \in {\S}  \text{ s.t. } h_f(\x, \y)  \geq 0\},
\]
with
\[
h_f : \R^{n+m} \to \R, \quad (\x,\y) \mapsto h_f(\x, \y) := - \|\y - f(\x)\|_2^2 = - \sum_{j=1}^m (y_j - {f_j}(\x))^2.
\]
Let us also define
\[
h : \R^m \to \R, \quad \y \mapsto h(\y) := \sup_{\x \in \S} h_f(\x,\y). 
\]

\begin{theorem}
\label{th:underH}
There exists a sequence of polynomials $(p_r)_{r \in \N} \subset \R[\y]$ such that $p_r(\y) \geq h_f(\x, \y)$ for all $r \in \N$, $\x \in \S$, $\y \in \B$ and such that
\begin{align}
\lim_{r \to \infty} \int  |p_r(\y) -  h(\y)  | \, \lambda_\B (d \y)  = 0 .
\end{align}
\end{theorem}
\begin{proof}
The result follows readily from~\cite[Theorem 3.1 (3.4)]{Las13} with the notations $\x \leftarrow \y$, $\y \leftarrow  \x$, $\K \leftarrow  \B \times \S$, $\K_\x \leftarrow  \S \neq \emptyset$ and $\underline{J_f} \leftarrow  - h$, which is lower semi-continuous.
\end{proof}
\begin{theorem}
\label{th:volH}
For each $r \in \N$,  define $\F_r := \{ \y \in \B : p_r(\y) \geq 0 \}$
where the sequence of polynomials $(p_r)_{r \in \N} \subset \R[\y]$ is as in Theorem~\ref{th:underH}. Then $\F_r \supset \F$ and one has
\begin{align}
\label{eq:volconv}
\lim_{r \to \infty} \vol (\F_r \backslash \F) = 0.
\end{align}
\end{theorem}
\begin{proof}
Let $r \in \N$. By assumption, one has $p_r(\y) \geq h_f(\x, \y)$ for all $\x \in \S, \y \in \B$. Thus, one has $p_r(\y) \geq h(\y)$ for all $\y \in \B$, which implies that $\F_r \supset \F$.

It remains to prove~\eqref{eq:volconv}.
Let us define $\F(k) := \{ \y \in \B : h(\y) \geq - 1/k \}$. 
First, we show that 
\begin{align}
\label{eq:DkfS}
\lim_{k \to \infty} \vol \F(k) = \vol \F.
\end{align}
For each $k \in \N$, one has $\F(k+1) \subseteq \F(k) \subseteq \B$, thus the sequence of indicator 
functions $(\mathbf{1}_{\F(k)})_{k \in \N}$ is non-increasing and bounded. Next, let us show that for all $\y \in \B$,
$\mathbf{1}_{\F(k)}(\y) \to \mathbf{1}_{\F}(\y)$, as $k \to \infty$:
\begin{itemize}
\item Let $\y \in \F$. By the inclusion $\F \subseteq  \F(k)$, for each $k \in \N$, $\mathbf{1}_{\F(k)}(\y) = \mathbf{1}_{\F}(\y) = 1$ and the result trivially holds.
\item Let $\y \in \B \backslash \F$, so there exists $\epsilon >0$ such that $h(\y)= - \epsilon$. Thus, there exists $k_0 \in \N$ such that for all $k \geq k_0$, $\y \in \B \backslash \F(k)$.
\end{itemize}
Hence, $\mathbf{1}_{\F(k)}(\y) \to \mathbf{1}_{\F}(\y)$ for each $\y \in \B$, as $k \to \infty$ (monotone non-increasing).
By the Monotone Convergence Theorem, $\mathbf{1}_{\F(k)}(\y) \to \mathbf{1}_{\F}(\y)$ for the $L_1$ norm on $\B$ and~\eqref{eq:DkfS} holds.

Next, we prove that for each $k \in \N$,
\begin{align}
\label{eq:DdleqDk}
\lim_{r \to \infty} \vol \F_r \leq  \vol \F(k).
\end{align}

By Theorem~\ref{th:underH} applied to the sequence $(p_r)_{r \in \N}$, one has  $\lim_{r \to \infty} \int|p_r(\y) - h(\y)  | \, \lambda_\B (d \y)  = 0$. Thus, by~\cite[Theorem 2.5.1]{Ash:72RAP}, the sequence $(p_r)_{r \in \N}$ converges to $h$ in measure, i.e. for every $\epsilon > 0$, 
\begin{align}
\label{eq:ash}
\lim_{r \to \infty} \vol ( \{ \y \in \B : |p_r(\y) - h(\y) | \geq \epsilon \} ) = 0.
\end{align}
For every $k \geq 1$, observe that:
\begin{eqnarray}
\label{eq:bndash}
\vol \F_r & = & \vol (\F_r  \cap \{\y \in \B : |p_r(\y) - h(\y) | \geq 1/k \} ) + \nonumber \\
                     &    & \vol (\F_r  \cap \{\y \in \B : |p_r(\y) - h(\y) | < 1/k \} ).            
\end{eqnarray}
It follows from~\eqref{eq:ash} that $\lim_{r \to \infty} \vol (\F_r \cap \{\y \in \B : |p_r(\y) - h(\y) | \geq 1/k \} ) = 0$. 
In addition, for all $r \in \N$,  
\begin{align}
\label{eq:bndHpk}
\vol (\F_r  \cap \{\y \in \B : |p_r(\y) - h(\y) | < 1/k \} ) \leq \vol ( \{\y \in \B : h(\y) \geq -1/k \} ) = \vol \F(k).
\end{align}
Using both~\eqref{eq:bndash} and~\eqref{eq:bndHpk}, and letting $r \to \infty$, yields~\eqref{eq:DdleqDk}.
Thus, we have the following inequalities:
\[
\vol  \F \leq \lim_{r \to \infty} \vol \F_r \leq  \vol \F(k).
\]
Using~\eqref{eq:DkfS} and letting $k \to \infty$ yields the desired result.
\end{proof}

\subsection{Practical computation using semidefinite programming}
\label{subsec:sdpexists}

In this section we show how the sequence of polynomials
of Theorems \ref{th:underH} and \ref{th:volH} can be computed in practice.
Define $r^{(1)}_{\min} := \max \{d, r_1, \dots, r_{n^\K} \}$.
For $r \geq r^{(1)}_{\min}$, consider the following hierarchy of semidefinite programs:
\begin{equation}
\label{eq:primalexists}
\begin{aligned}
p^*_r := \inf\limits_q \quad & \displaystyle\sum_{\beta \in \N_{2 r}^m} q_\beta  z^\B_\beta \\			
\text{s.t.} \quad & q - h_f \in \Q_r(\K), \\
         \quad & q \in \R_{2 r}[\y].
\end{aligned}
\end{equation}		
The semidefinite program dual of \eqref{eq:primalexists} reads:
\begin{equation}
\label{eq:dualexists}
\begin{aligned}
d^*_r := \sup\limits_{\z} \quad &\ell_{\z}(h_f) \\			
\text{s.t.} \quad & \M_r(\z) \succeq 0, \\
\quad & \M_{r - r_j^\K}(g_j \z) \succeq 0,  \quad j=1,\ldots,n^\K, \\
& \ell_{\z}(\y^\beta) = z^\B_\beta,  \quad \forall \beta \in \N_{2 r}^m . \\
\end{aligned}
\end{equation}	
\begin{theorem}
\label{th:sdpexists}
Let $r \geq r^{(1)}_{\min}$ and suppose that Assumption~\ref{hyp:archimedean} holds. Then: 
\begin{enumerate}
\item $p^*_r=d^*_r$, i.e. there is no duality gap between the semidefinite program~\eqref{eq:primalexists} and its dual~\eqref{eq:dualexists}.
\item The semidefinite program~\eqref{eq:dualexists} has an optimal solution. In addition, if $\S$ has nonempty interior, then the semidefinite program~\eqref{eq:primalexists} has an optimal solution $q_r$, and the sequence $(q_r)_{r \in \N}$
converges to $h$ in $L_1$ norm on $\B$:
\begin{equation}
\label{eq:qdcvg}
\lim_{r \to \infty} \int  |q_r(\y) - h(\y) | \, \lambda_\B (d \y) = 0.
\end{equation}
\item Defining the set
\[
\F^1_r := \{ \y \in \B : q_r(\y) \geq 0 \}
\]
it holds that
\[
\F^1_r \supset \F
\]
and
\[
\lim_{r \to \infty} \vol (\F^1_r \backslash \F) = 0.
\]
\end{enumerate}
\end{theorem}

\begin{proof}
\label{pr:qdcvg}

\begin{enumerate}
\item Let $\mathcal{D}_r$ (resp. $\mathcal{D}_r^*$) stand for the feasible (resp. optimal) solution set of the semidefinite program~\eqref{eq:dualexists}. First, we prove that $\mathcal{D}_r \neq \emptyset$.
Let $\z = (z_{\alpha \beta})_{(\alpha, \beta) \in \N_{2 r}^{n+m}}$ be the sequence moments of $\lambda_\K$, the Lebesgue
measure on $\K=\S\times\B$ . Since the measure is supported on $\K$, the semidefinite constraints $\M_{r}(\z) \succeq 0$, $\M_{r - r_j^\K}(g_j \z) \succeq 0$, $j=1, \dots, n^\K$ are satisfied. By construction, the marginal of $\lambda_\K$ on $\B$ is $\lambda_\B$ and the following equality constraints are satisfied: $\ell_{\z}(\y^\beta) = z^\B_\beta$ for all $\beta \in \N_{2 r}^m$.
Thus, the finite sequence $\z$ lies in $\mathcal{D}_r \neq \emptyset$.

Note that Assumption~\ref{hyp:archimedean} implies that the semidefinite constraints $\M_{r - 1}(g^\S \z) \succeq 0$ and $\M_{r - 1}(g^\B \z) \succeq 0$ both hold. Thus, the first diagonal elements of $\M_{r - 1}(g^\S \z)$ and $\M_{r - 1}(g^\B \z)$ are nonnegative,
and since  $\ell_{\z}(1)  = z^\B_0$, it follows that $\ell_{\z}(x_i^{2 k}) \leq (N^\S)^k z^\B_0$, $i=1,\ldots,n$
and $\ell_{\z}(y_j^{2 k}) \leq (N^\B)^k z^\B_0$, $j=1,\dots,m$, $k=0,\ldots,r$ and 
we deduce from~\cite[Lemma 4.3, p. 111]{LasserreNetzer07SOS} that $|z_{\alpha \beta}|$ is
bounded for all $(\alpha, \beta) \in \N_{2 r}^{n+m}$. Thus, the feasible set $\mathcal{D}_r$ is compact as closed and bounded. Hence, the set $\mathcal{D}_r^*$ is nonempty and bounded. The claim then follows from the sufficient condition of strong duality in~\cite{Trnovska:2005:SDP}.

\item  Assume that $\S$ has nonempty interior, so $\K = \S \times \B$ has also nonempty interior. Thus, the feasible solution $\z$  (defined above) satisfies  $\M_{r}(\z) \succ 0$, $\M_{r - r_j^\K}(g_j \z) \succ 0$, $j=1, \dots, n^\K$, which implies that Slater's condition holds for~\eqref{eq:dualexists}. Note also that the semidefinite program \eqref{eq:primalexists} has the trivial feasible solution $q=0$
since $-h_f$ is SOS by construction.
As a consequence of a now standard result of duality in semidefinite programming (see e.g.~\cite{Vandenberghe94sdp}), the semidefinite program~\eqref{eq:primalexists} has an optimal solution $q_r \in \R_{2 r}[\y]$.

Let us consider a sequence of polynomials $(p_k)_{k\in\N} \subset \R[\y]$ as in Theorem~\ref{th:underH}. Now, fix $\epsilon > 0$. By Theorem~\ref{th:underH}, there exists $k_0 \in \N$  such that 
\begin{align}
\label{eq:qdcvg1}
\int |p_k(\y) - h(\y) | \lambda_\B (d \y)\leq \epsilon / 2,
\end{align}
for all $k \geq k_0$. Then, observe that the polynomial $p^\epsilon_k := p_k + \epsilon / (2 \vol \B)$ satisfies $p^\epsilon_k(\y) - h_f(\x, \y)  > 0$, for all $\x \in \S$, $\y \in \B$.  
For $r \in \N$ large enough, as a consequence of Putinar's Positivstellensatz (e.g.~\cite[Section 2.5]{lasserre2009moments}), there exist $s_0, \dots, s_{n^\K} \in \Sigma[\x, \y]$ such that 
\[
p^\epsilon_k(\y) - h_f(\x, \y) = \sum_{j = 0}^{n^\K} s_j(\x, \y) g_j(\x, \y),
\]
with $\deg (s_j g_j) \leq 2 r$ for $j = 0, \dots, n^\K$. And so, $p^\epsilon_k-h_f$ lies in the $r$-truncated quadratic module $\Q_r(\K)$, which implies that $p^\epsilon_k$ is a feasible solution for Problem~\eqref{eq:primalexists}.
\if{
Let us consider a sequence of polynomials $(p_r)_{r\in\N} \subset \R[\y]$ as in Theorem~\ref{th:underH}. Now, fix $\epsilon > 0$. By Theorem~\ref{th:underH}, there exists $r_0 \in \N$  such that 
\begin{align}
\label{eq:qdcvg1}
\int |p_r(\y) - h(\y) | \lambda_\B (d \y)\leq \epsilon / 2,
\end{align}
for all $r \geq r_0$. Then, observe that the polynomial $p^\epsilon_r := p_r + \epsilon / (2 \vol \B)$ satisfies $p^\epsilon_r(\y) - h_f(\x, \y)  > 0$, for all $\x \in \S$, $\y \in \B$.  Let $r \geq \max \{\lceil (\deg p_r) / 2 \rceil, r^{(1)}_{\min} \}$.
As a consequence of Putinar's Positivstellensatz (e.g.~\cite[Section 2.5]{lasserre2009moments}), there exist $s_0, \dots, s_{n^\K} \in \Sigma[\x, \y]$ such that 
\[
p^\epsilon_r(\y) - h_f(\x, \y) = \sum_{j = 0}^{n^\K} s_j(\x, \y) g_j(\x, \y),
\]
with $\deg (s_j g_j) \leq 2 r$ for $j = 0, \dots, n^\K$. And so, $p^\epsilon_r-h_f$ lies in the $r$-truncated quadratic module $\Q_r(\K)$, which implies that $p^\epsilon_r$ is a feasible solution for Problem~\eqref{eq:primalexists}. 
}\fi
Hence, $q_r$ being an optimal solution of problem~\eqref{eq:primalexists}, the following holds:
\begin{align}
\label{eq:qdcvg2}
\int q_{r} (\y) \lambda_\B (d \y) \leq  \int p^\epsilon_{k} (\y) \lambda_\B (d \y)  = \int [ p_k(\y) + \epsilon / (2 \vol (\B)) ]  \lambda (d \y)   .
\end{align}
Combining~\eqref{eq:qdcvg1} and~\eqref{eq:qdcvg2} yields $\displaystyle\int | q_{r}(\y) - h(\y)  | \lambda (d \y) \leq \epsilon$, concludes the proof.

\item This is a consequence of Theorem \ref{th:volH}.
\end{enumerate}
\end{proof} 
\section{Method 2: support of image measures}
\label{sec:imagemeasure}

Given a compact set $\A \subset \R^n$, let $\mathcal{M}(\A)$  stand for the vector space of finite signed Borel measures supported on $\A$, understood as functions from the Borel sigma algebra $\mathcal{B}(\A)$ to the real numbers. Let $\mathcal{C}(\A)$ stand for the space of continuous functions on $\A$, equipped with the sup-norm (a Banach space).
Since $\A$ is compact,  the topological dual
(i.e. the set of continuous linear functionals) of $\mathcal{C}(\A)$ (equipped with the sup-norm), denoted by $\mathcal{C}(\A)'$,
is (isometrically isomorphically identified with) $\mathcal{M}(\A)$ equipped with the total variation norm, denoted by $\| \cdot\|_{\text{TV}}$.  The cone of non-negative elements of $\mathcal{C}(\A)$, resp. $\mathcal{M}(\A)$, is denoted by $\mathcal{C}_+(\A)$,
resp. $\mathcal{M}_+(\A)$. 
The topology in $\mathcal{C}_+(\A)$ is the strong topology of uniform convergence while the topology in $\mathcal{M}_+(\A)$ is the weak-star topology (see~\cite[Chapter IV]{alexander2002course} or~\cite[Section 5.10]{Luenberger97} for more background on weak-star topology).

Recall that $\lambda_\B$ stand for the Lebesgue measure on $\B$. If $\mu,\nu \in \mathcal{M}(\A)$,
the notation $\mu \ll \nu$ stands for $\mu$ being absolutely continuous w.r.t. $\nu$, whereas
the notation $\mu \leq \nu$ means that $\nu-\mu \in \mathcal{M}_+(\A)$. For background on functional analysis and
measure spaces see e.g. \cite[Section 21.5]{Royden}.
\subsection{LP primal-dual conic formulation}
\label{sec:lpprimaldual}
Given a polynomial application $f : \S \to \B$, the pushforward or image map
\[
f_\# : \mathcal{M}(\S) \to \mathcal{M}(\B)
\]
is defined such that
\[
f_\# \mu_0 (\A) := \mu_0 (\{ \x \in \S : f(\x) \in \A \})
\]
for every set $\A \in \mathcal{B}(\B)$ and every measure $\mu_0 \in \mathcal{M}(\S)$. The measure $f_\# \mu_0 \in \mathcal{M}(\B)$ is then called the \textit{image measure} of $\mu_0$ under $f$, see e.g. \cite[Section 1.5]{AFP00}.


To approximate the image set $\F=f(\S)$, one considers the infinite-dimensional linear programming (LP) problem:
\begin{equation}
\label{eq:lpmeasure}
\begin{aligned}
p^* := \sup\limits_{\mu_0, \mu_1, \hat{\mu}_1} \quad & \int  \mu_1 \\			
\text{s.t.} \quad & \mu_1 + \hat{\mu}_1 = \lambda_\B,\\
\quad & \mu_1 = f_\# \mu_0,\\
\quad & \mu_0 \in  \mathcal{M}_+(\S), \quad  \mu_1, \hat{\mu}_1 \in  \mathcal{M}_+(\B),\\
\end{aligned}
\end{equation}
In the above LP, by definition of the image measure, $\mu_1$ exists
whenever $\mu_0$ is given. The following result gives conditions for $\mu_0$
to exist whenever $\mu_1$ is given.

\begin{lemma}\label{lem:farkas}
Given a measure $\mu_1 \in \mathcal{M}_+(\B)$, there is a measure $\mu_0 \in \mathcal{M}_+(\S)$
such that $f_\#\mu_0 = \mu_1$ if and only if there is no continuous function $v \in \mathcal{C}(\B)$
such that $v(f(\x)) \geq 0$ for all $\x \in \S$ and $\int v(\y) d\mu_1(\y) < 0$.
\end{lemma}

\begin{proof}
This follows from \cite[Theorem 6]{CravenKoliha77} which is an extension
to locally convex topological spaces of the celebrated
Farkas Lemma in finite-dimensional linear optimization.
One has just to verify that the image cone $f_\#(\mathcal{M}_+(\S))=\{f_\#\mu_0 : \mu_0 \in \mathcal{M}_+(\S)\}$
is closed in the weak-star topology $\sigma(\mathcal{M}(\B),\mathcal{C}(\B))$ of
$\mathcal{M}(\B)$.This, in turn, follows from continuity of $f$
and compactness of $\S$.
\end{proof}

\if{
\begin{corollary}\label{cor:farkas}
Given any measure $\mu_1  \in \mathcal{M}_+(\B)$ admissible for LP (\ref{eq:lpmeasure}),
there is a measure $\mu_0 \in \mathcal{M}_+(\S)$ such that $f_\#\mu_0 = \mu_1$.
\end{corollary}

\begin{proof}
Since $\mu_1 \in \mathcal{M}_+(\B)$, $\mu_1+\hat{\mu_1} = \lambda_\B$ and $\mathrm{spt}\:\mu_1 \subseteq \F$,
it holds $0 \leq \mu_1 \ll \lambda_\F$ and by Radon-Nikod\'ym there exists a function $q_1 \in L_1(\lambda_\F)$
such that $d\mu_1(\y) = q_1(\y)d\lambda_\F(\y)$ and $q_1(\y) \geq 0$ for all $\y \in \F$.
The claim follows then from Lemma \ref{lem:farkas}
since it is impossible to find a function $v \in \mathcal{C}(\B)$ such that
$v(\y) \geq 0$ for all $\y \in \F$ while satisfying $\int v(\y)q_1(\y)d\lambda_\F(\y) < 0$. 
\end{proof}
}\fi

\begin{lemma}
\label{th:lpmeasure}
LP~\eqref{eq:lpmeasure} admits an optimal solution $(\mu_0^*, \mu_1^*, \hat{\mu}_1^*)$. Moreover, $\mu_1^*=\lambda_\F$ and the optimal value of LP~\eqref{eq:lpmeasure} is $p^* =  \vol \F$.
\if{
Let $\mu_1^*$ be an optimal solution of LP~\eqref{eq:lpmeasure}.
Then $\mu_1^*=\lambda_\F$ and $p^* =  \vol \F$.
}\fi
\end{lemma}

\begin{proof}
\label{pr:lpmeasure}
First, we prove that for 
$\mu_1^*=\lambda_\F \in \mathcal{M}_+(\B)$, there is a measure $\mu_0^* \in \mathcal{M}_+(\S)$ such that $f_\#\mu_0^* = \mu_1^*$.
Indeed, by Radon-Nikod\'ym there exists a function $q_1 \in L_1(\lambda_\F)$
such that $d\mu_1^*(\y) = q_1(\y)d\lambda_\F(\y)$ and $q_1(\y) \geq 0$ for all $\y \in \F$. The claim follows then from Lemma \ref{lem:farkas}
since it is impossible to find a function $v \in \mathcal{C}(\B)$ such that
$v(\y) \geq 0$ for all $\y \in \F$ while satisfying $\int v(\y)q_1(\y)d\lambda_\F(\y) < 0$.
Define $\hat{\mu}_1^* = \lambda_\B - \mu_1^*$, then $(\mu_0^*, \mu_1^*, \hat{\mu}_1^*)$ is admissible for LP~\eqref{eq:lpmeasure}. Exactly as in the proof of \cite[Theorem 3.1]{HLS08vol}, one finally shows that this triplet is optimum as well as uniqueness of $\mu_1^*$, yielding the optimal value $p^* =  \vol \F$.
\if{
From Corollary \ref{cor:farkas}, the choice $\mu^*_1=\lambda_\F  \in \mathcal{M}_+(\B)$ is admissible for LP (\ref{eq:lpmeasure})
since there exists $\mu_0^* \in \mathcal{M}_+(\S)$ such that $f_\#\mu_0^* = \mu_1^*$
and we can enforce $\hat{\mu}_1^* = \lambda_\B - \mu_1^*$.
}\fi
\if{
Optimality of $\mu^*_1=\lambda_\F$ then follows exactly as in the proof of \cite[Theorem 3.1]{HLS08vol}.
}\fi
\end{proof}
Next, we express problem~\eqref{eq:lpmeasure} as an infinite-dimensional conic problem
on appropriate vector spaces. By construction, a feasible solution of problem~\eqref{eq:lpmeasure} satisfies:
\begin{align}
\int_\B v(\y) \, \mu_1(d \y) - \int_\S v(f(\x)) \, \mu_0(d \x) & =  0  , \label{eq:testv} \\
\int_\B w(\y) \, \mu_1(d \y) + \int_\B w(\y) \, \hat{\mu}_1(d \y) & =  \int_\B w(\y) \, \lambda (d \y)  , \label{eq:testw}
\end{align}
for all continuous test functions $v, w \in \mathcal{C}(\B)$.

Then, we cast problem~\eqref{eq:lpmeasure} as a particular instance of a primal  LP in the canonical form given in~\cite[7.1.1]{alexander2002course}:
\begin{equation}
\label{eq:lpprimal}
\begin{aligned}
p^* = \sup\limits_{x} \quad & \langle x, c \rangle_1 \\
\text{s.t.} \quad & \mathcal{A} \, x = b  , \\
\quad & x \in E_1^+ , \\
\end{aligned}
\end{equation}
with
\begin{itemize}
\item the vector space $E_1 := \mathcal{M}(\S) \times \mathcal{M}(\B)^2$;
\item the vector space $F_1 := \mathcal{C}(\S) \times \mathcal{C}(\B)^2$;
\item the duality $\langle \cdot , \cdot \rangle_1 : E_1 \times F_1 \to \R$, given by the integration of continuous functions against Borel measures, since $E_1 = F_1'$;
\item the decision variable $x := (\mu_0, \mu_1, \hat{\mu}_1) \in E_1$ and the reward $c := (0, 1, 0) \in F_1$;
\item $E_2 := \mathcal{M}(\B)^2$, $F_2 := \mathcal{C}(\B)^2$ and the right hand side vector $b := (0, \lambda) \in E_2=F_2'$;
\item the linear operator $\mathcal{A} : E_1 \to E_2$ given by
\[
\mathcal{A} \, (\mu_0, \mu_1, \hat{\mu}_1) := \left[\begin{array}{cc}
-f_\#\mu_0 + \mu_1 \\  \mu_1+\hat{\mu}_1 
\end{array}\right].
\]
\end{itemize}
Notice that all spaces $E_1$, $E_2$ (resp. $F_1$, $F_2$) are equipped with the weak topologies $\sigma(E_1, F_1)$, $\sigma(E_2, F_2)$ (resp. $\sigma(F_1, E_1)$, $\sigma(F_2, E_2)$). Importantly, $\sigma(E_1, F_1)$ is the weak-star topology (since $E_1 = F_1'$). Observe that $\mathcal{A}$ is continuous with respect to the weak topology, as $\mathcal{A}'(F_2) \subset F_1$.

With these notations,  the dual LP in the canonical form given in~\cite[7.1.2]{alexander2002course} reads:
\begin{equation}
\label{eq:lpdual}
\begin{aligned}
d^* = \inf\limits_{y} \quad & \langle b, y \rangle_2 \\
\text{s.t.} \quad & \mathcal{A}' \, y  -c\:\in \mathcal{C}_+(\B)^2\\
\end{aligned}
\end{equation}
with
\begin{itemize}
\item the dual variable $y := (v, w) \in E_2$;
\item the (pre)-dual cone $\mathcal{C}_+(\B)^2$, whose dual is $E_1^+$;
\item the duality pairing $\langle \cdot , \cdot \rangle_2 : E_2 \times F_2 \to \R$, with $E_2 = F_2'$;
\item the adjoint linear operator $\mathcal{A}' : F_2 \to F_1$ given by
\[
\mathcal{A}' \, (v,w) := \left[\begin{array}{c}
-v \circ f \\ v + w \\ w
\end{array}\right].
\]
\end{itemize}
Using our original notations, the dual LP of problem~\eqref{eq:lpmeasure} then reads:
\begin{equation}
\label{eq:lpcont}
\begin{aligned}
d^* := \inf\limits_{v, w} \quad & \int w(\y) \, \lambda_\B (d  \y) \\			
\text{s.t.} \quad & v(f(\x)) \geq 0 , \quad \forall \x \in \S  , \\
\quad & w(\y) \geq 1 + v(\y) , \quad \forall \y \in \B  , \\
\quad & w(\y) \geq 0 , \quad \forall \y \in \B  , \\
\quad & v, w \in \mathcal{C}(\B)  . \\
\end{aligned}
\end{equation}

\begin{theorem}
\label{th:zerogap}
There is no duality gap between problem~\eqref{eq:lpmeasure} and problem~\eqref{eq:lpcont}, i.e. $p^*=d^*$.
\end{theorem}
\begin{proof}
\label{pr:zerogap}
This theorem follows from the ``zero duality gap'' result from~\cite[Theorem 7.2]{alexander2002course}, if one can prove that the cone
\begin{equation}
\label{eq:coneA}
\mathcal{A} \, (E_1^+) := \{ (\mathcal{A} \, x, \langle x, c \rangle_1) : x \in E_1^+ \}
\end{equation}
is closed in $E_2 \times \R$. 
To do so, let us consider a sequence  $(x^{(k)}) = (\mu_0^{(k)},\mu_1^{(k)},\hat{\mu}_1^{(k)})) \subset E_1^+$ such that $\mathcal{A} \, x^{(k)} \to s = (s_1, s_2)$ and $\langle x^{(k)}, c \rangle_1 \to t$. Let us prove that $(s, t) = (\mathcal{A} \, x^*, \langle x^*, c \rangle_1)$ for some $x^* = (\mu_0^*,\mu_1^*,\hat{\mu}_1^*) \in E_1^+$. As $c = (0, 1, 0)$, one has $\|\mu_1^{(k)}\|_{\text{TV}} = \int_\B \mu_1^{(k)} (d \y) \to t (\geq 0)$, thus  $\sup_k \|\mu_1^{(k)}\|_{\text{TV}}  < \infty$. Therefore there is a subsequence
(denoted by the same indices) $(\mu_1^{(k)})$ which converges to $\mu_1^* \in \mathcal{M}_+(\B)$ for
the weak-star topology. In particular $\Vert\mu_1^*\Vert_{TV}=t$. Hence from $\hat{\mu}_1^{(k)}+\mu_1^{(k)} \to s_2$ one deduces that $\hat{\mu}_1^{(k)} \to s_2-\mu_1^*$ for the weak-star topology. But then we also have $-f_\#\mu_0^{(k)} \to s_1-\mu_1^*$ in the weak-star topology
of $\mathcal{M}(\B)$. Therefore $-s_1+\mu_1^*$ is a positive measure. So let $\mu_0^*$ be such that $f_\#\mu_0^* = -s_1+\mu_1^*$ guaranteed to exist 
since we have seen that $f_\#(\mathcal{M}_+(S))$ is weak-star closed. Then we have $\mathcal{A}(x^*) = (s_1,s_2)$ and $\langle x^*,c\rangle =t$, the desired result.
\end{proof}
\subsection{Practical computation using semidefinite programming}
\label{sec:momentsos}
For each $r \geq r^{(2)}_{\min} := \max \{\lceil r^\S_1 / d \rceil, \ldots, \lceil r^\S_{n^\S} / d \rceil, r^\B_1, \ldots, r^\B_{n^\B}\}$,
let $\z_0 = (z_{0 \beta})_{\beta \in \N_{2 r}^m}$ be the finite sequence of moments up to degree $2 r$
of measure $\mu_0$. Similarly, let $\z_1$ and $\hat{\z}_1$ stand for the sequences of moments up to degree $2 r$, respectively associated with $\mu_1$ and $\hat{\mu}_1$.
Problem~\eqref{eq:lpmeasure} can be relaxed with the following semidefinite program:
\begin{equation}
\label{eq:lprelax}
\begin{aligned}
p^*_r := \sup\limits_{\z_0, \z_1, \hat{\z}_1} \quad & z_{1 0} \\			
\text{s.t.} \quad & z_{1 \beta} + \hat{z}_{1 \beta} = z^{\B}_\beta, L_{\z_0}(f(\x)^\beta) = z_{1 \beta},
\quad \forall \beta \in N_{2 r}^m, \\
\quad & \M_{r d -r_j^\S} (g_j^\S \z_0) \succeq 0, \quad j = 0,\dots, n^\S, \\
\quad & \M_{r - r_j^\B}(g_j^\B \z_1) \succeq 0 , \M_{r - r_j^\B} (g_j^\B  \hat{\z}_1) \succeq 0, \quad j = 0,\dots, n^\B.
\end{aligned}
\end{equation}
Consider also the following semidefinite program, which is a strengthening of problem~\eqref{eq:lpcont}
and also the dual of problem~\eqref{eq:lpstrength}:
\begin{equation}
\label{eq:lpstrength}
\begin{aligned}
d^*_r := \inf\limits_{v, w} \quad & \sum_{\beta \in \N_{2 r}^m} w_\beta z^\B_\beta \\			
\text{s.t.} \quad & v \circ f \in \Q_{r d}(\S), \\
\quad & w -1-v \in \Q_r(\B), \\
\quad & w \in \Q_r(\B),\\
\quad & v, w \in \R_{2 r}[\y]. \\
\end{aligned}
\end{equation}
\begin{theorem}
\label{th:lpstrength}
Let $r \geq r^{(2)}_{\min}$ and suppose that both $\F$ and $\B \backslash \F$ have nonempty interior and that Assumption~\ref{hyp:archimedean} holds. Then:
\begin{enumerate}
\item $p^*_r=d^*_r$, i.e. there is no duality gap between the semidefinite program (\ref{eq:lprelax}) and its dual (\ref{eq:lpstrength}).
\item The semidefinite program~\eqref{eq:lpstrength} has an optimal solution $(v_r, w_r) \in \R_{2 r}[\y] \times \R_{2 r}[\y]$, and the sequence $(w_r)$ converges to $\mathbf{1}_\F$ in $L_1$ norm on $\B$:
\begin{equation}
\label{eq:wdcvg}
\lim_{r \to \infty} \int |w_r(\y)-\mathbf{1}_{\F} (\y)| \, \lambda_\B (d \y)  = 0.
\end{equation}
\item Defining the set
\[
\F^2_r := \{ \y \in \B : w_r (\y) \geq 1 \}
\]
its holds that
\[
\F^2_r \supset \F
\]
and
\[
\lim_{r \to \infty} \vol (\F^2_r \backslash \F) = 0.
\]
\end{enumerate}
\end{theorem}
\begin{proof}
\label{pr:lpstrength}
~

\begin{enumerate}
\item 
Let $\mu_1=\lambda_\F$, let $\mu_0$ be such that $f_\#\mu_0 = \mu_1$ as in Lemma \ref{lem:farkas},
and let $\hat{\mu}_1 = \lambda_\B - \mu_1$ so that $(\mu_0,\mu_1,\hat{\mu}_1)$ is feasible for LP (\ref{eq:lpmeasure}).
Given $r \geq r^{(2)}_{\min}$, let $\z_0$, $\z_1$ and $\hat{\z}_1$ be the sequences of moments up to degree $2 r$
of $\mu_0$, $\mu_1$ and $\hat{\mu}_1$, respectively. Clearly, $(\z_0, \z_1, \hat{\z}_1)$ is feasible for program~\eqref{eq:lprelax}.
Then, as in the proof of the first item of Theorem~\ref{th:sdpexists}, the optimal solution set of the program~\eqref{eq:lprelax} is nonempty and bounded, which by~\cite{Trnovska:2005:SDP} implies that there is no duality gap between the semidefinite program~\eqref{eq:lpstrength} and its dual~\eqref{eq:lprelax}. 

\item
Now, one shows that $(\z_0, \z_1, \hat{\z}_1)$ is strictly feasible for program~\eqref{eq:lprelax}.
Using the fact that 
\begin{enumerate}
\item $\F$ (resp. $\B \backslash \F$) has nonempty interior,
\item $\z_1$  (resp. $\hat{\z_1}$) is the moment sequence of $\mu_1$ (resp. $\hat{\mu}_1$),
\end{enumerate}
one has  $\M_r(g_j^\B \z_1) \succ 0$ (resp. $\M_r(g_j^\B \hat{\z}_1) \succ 0$), for each $j = 0,\dots, n^\B$. Moreover, $\M_r(g_j^\S \z_0) \succ 0$, for all  $j = 0,\dots, n^\S$. Otherwise, assume that there exists a nontrivial vector $\q$ such that $\M_r(g_j^\S \z_0) \, \q  = 0$ for some $j$. As $\F$ has nonempty interior, it contains an open set $\A \subset \R^m$. By continuity of $f$, the set $f^{-1}(\A) :=\{\x \in \S : f(\x) \in \A\}$ is an open set of $\S$ and $\mu_0 (f^{-1}(\A)) = \mu_1 (\A) > 0$. Then,
\[
0 = \langle \q, \M_r(g_j^\S \z_0) \, \q \rangle = \int_{\S} q(\x)^2 \, g_j^\S(\x) \, d \mu_0 (\x)  \geq \int_{f^{-1}(\A)} q(\x)^2 \, g_j^\S(\x) \, d \mu_0 (\x) ,
\]
which yields $q(\x)^2 \, g_j^\S(\x) = 0$ on the open set $f^{-1}(\A)$, leading to a contradiction.

Therefore, as for the proof of the second item of Theorem~\ref{th:sdpexists}, we conclude that the semidefinite program~\eqref{eq:lpstrength} has an optimal solution $(v_r, w_r) \in \R_{2 r}[\y] \times \R_{2 r}[\y]$.

Next, one proves that there exists a sequence of polynomials $(w_k)_{k \in \N}
\subset \R[\y] $ such that $w_k(\y) \geq \mathbf{1}_{\F} (\y)$, for all $\y \in \B$ and such that
\begin{align}
\label{eq:overwdcvg}
\lim_{k \to \infty} \int  |w_k(\y) - \mathbf{1}_{\F} (\y) | \, \lambda_\B (d \y)  = 0.
\end{align}
The set $\F$ being closed, the indicator function $\mathbf{1}_{\F}$ is upper semi-continuous and bounded, so there exists a non-increasing sequence of bounded continuous functions $h_k : \B \to \R$ such that $h_k (\y) \downarrow \mathbf{1}_{\F} (\y)$, for all $\y \in \B$, as $k \to \infty$. Using the Monotone Convergence Theorem, $h_k \to \mathbf{1}_{\F}$ for the $L_1$ norm. By the Stone-Weierstrass Theorem, there exists a sequence of polynomials $(w_k')_{k \in \N} \subset \R[\y]$, such that $\sup_{\y \in \B} |w_k'(\y) - h_k(\y)| \leq 1/k$. The polynomial $w_k := w_k' + 1/k$ satisfies $w_k > h_k \geq \mathbf{1}_{\F}$ and~\eqref{eq:overwdcvg} holds.

Let us define $\tilde{w}_k := w_k + \epsilon / (2 \vol \B)$, $\tilde{v}_k := w_k -1$.
Next, for $r \in \N$ large enough, one proves that $(\tilde{v}_k, \tilde{w}_k)$ is a feasible solution of~\eqref{eq:lpstrength}. 
Using the fact that $\tilde{w}_k > w_k > \mathbf{1}_{\F}$, one has $\tilde{w}_k \in \Q_r(\B)$, as a consequence of Putinar's Positivstellensatz. 
For each $\x \in \S$,  $\tilde{v}_k (f(\x)) = w_k(\y) - 1 > \mathbf{1}_{\F} (f(\x)) - 1 + \epsilon / (2 \vol \B) > 0$, so $\tilde{v}_k \circ f$ lies in $\Q_{r d}(\S)$.
Similarly, $\tilde{w}_k - \tilde{v}_k - 1 \in \Q_r(\B)$. Then, one concludes using the same arguments as for~\eqref{eq:qdcvg} in the proof of the second item of Theorem~\ref{th:sdpexists}.
\item Let $\y \in \F$. There exists $\x \in \S$ such that $\y = f(\x)$. Let $(v_r, w_r) \in \R_{2 r}[\y] \times \R_{2 r}[\y]$ be an optimal solution of~\eqref{eq:lpstrength}. By feasibility, $w_r(\y) -1 \geq v_r(\y) = v_r(f(\x)) \geq 0$. Thus, $\F^2_r \supset \F$. Finally, the proof
of the convergence in volume is analogous to the proof of~\eqref{eq:volconv} in Theorem~\ref{th:volH}.
\end{enumerate}
\end{proof}
\section{Computational considerations}
\label{sec:compl}
\subsection{Complexity analysis and lifting strategy}
\subsubsection{Method 1}
First, consider the semidefinite program~\eqref{eq:dualexists} of Method 1. For $r \geq r^{(1)}_{\min}$, the number of variables $n^{(1)}$ (resp. size of semidefinite matrices $m^{(1)}$) of problem~\eqref{eq:dualexists} satisfies:
\[ n^{(1)} \leq  \binom{n + m + 2 r }{2 r}   . \]
Problem~\eqref{eq:dualexists} involves $(n^\B+ n^\S + 1)$ semidefinite constraints of size $m^{(1)}$ bounded as follows:
\[ m^{(1)} \leq  \binom{n + m + r }{r}   . \]
\subsubsection{Method 2}
\label{sec:lift2}
Now, consider the  semidefinite program~\eqref{eq:lprelax}. 
For $r \geq r^{(2)}_{\min}$, the number of  variables $n^{(2)}$ of Problem~\eqref{eq:lprelax} satisfies:
\[ n^{(2)} \leq  \binom{n + 2 r  d}{2 r d} +  2 \binom{m + 2 r}{2 r} . \]
Problem~\eqref{eq:lprelax} also involves $(n^\S+1)$ semidefinite constraints of size at most $\binom{n + r d}{r d}$ and $2 (n^\B+1)$ semidefinite constraints of size at most $\binom{m + r}{r}$. 

Due to the dependence on the degree $d$ of the polynomial application, one observes that the number of variables (resp. constraints) can quickly become large if $d$ is not small.
An alternative formulation to limit the blowup of these relaxations is obtained by considering $y_1, \dots, y_m$ as ``lifting'' variables, respectively associated with $f_1, \dots, f_m$, together with the following $2m$ additional constraints: 
\[
g^\S_{n_\S+j}(\x, \y) := y_j - f_j(\x)  , \quad g^\S_{n_\S+2 j}(\x, \y) := f_j(\x) - y_j,  \quad j = 1,\dots, m.
\]

By considering the basic compact semi-algebraic set $\Slift \subset \R^{n+m}$ given by
\begin{equation}
\label{eq:defSlift}
\Slift := \{(\x, \y) \in \K : g^\S_{n^\S+1}(\x, \y) \geq 0, \dots, g^\S_{n^\S+2m}(\x, \y) \geq 0 \} , 
\end{equation}
problem~\eqref{eq:lpstrength} becomes:
\begin{equation}
\begin{aligned}
\inf\limits_{v, w} \quad & \sum_{\beta \in \N_{2 r}^m} w_\beta z^\B_\beta \\		
\text{s.t.} \quad & v \in \Q_{r}(\Slift)  , \\
\quad & w -1 - v\in \Q_r(\B) , \\
\quad & w \in \Q_r(\B),\\
\quad & v, w \in \R_{2 r}[\y]  , \\
\end{aligned}
\end{equation}
which is actually equivalent to the following problem:
\begin{equation}
\label{eq:lpstrengthlift}
\begin{aligned}
\inf\limits_{w} \quad & \sum_{\beta \in \N_{2 r}^m} w_\beta  z^\B_\beta \\		
\text{s.t.} \quad & w - 1 \in \Q_{r}(\Slift)  , \\
\quad & w \in \Q_r(\B)  , \\
\quad & w \in \R_{2 r}[\y]  . \\
\end{aligned}
\end{equation}

The minimal relaxation order of problem~\eqref{eq:lpstrengthlift} is $\dlift^{(2)}
 := \max \{\lceil\frac{d}{2}\rceil, r^\S_1, \ldots, r^\S_{n^\S}, r^\B_1, \ldots, r^\B_{n^\B}\}$
and the number of variables  $\nlift^{(2)}$ is bounded as follows:
\[ \nlift ^{(2)} \leq  \binom{n + m + 2 r}{2 r} + \binom{m + 2 r}{2 r}  . \]
Problem~\eqref{eq:lpstrengthlift} involves $(n^\S+ 2 m + 1)$ semidefinite constraints of size at most $\binom{n + m + r}{r}$ and $(n^\B+1)$ semidefinite constraints of size at most $\binom{m + r}{r}$. 
When $m$ is small and $d$ is large, this seems to be a suitable choice to reduce the computational cost of the semidefinite program~\eqref{eq:lprelax}. Experimental results described further (see Table~\ref{table:image} in Section~\ref{sec:polbenchs} and Table~\ref{table:saimage} in Section~\ref{sec:saimage}) agree with this observation.
\subsection{Exploiting sparsity}
As explained above, both Method 1 and Method 2 are computationally demanding in general. However, if the polynomials $\x \mapsto f_j(\x)$, $(j=1, \dots, m)$ have some structured sparsity, then one can still exploit sparsity in a way similar to the one described in~\cite{Waki06sumsof,Las06Sparse} to handle problems in higher dimensions. In particular, let $\{1, \dots, n\}$ be the union $\bigcup_{j=1}^{m} I_j$ of subsets $I_j \subseteq \{1, \dots, n\}$
and assume that for each $j=1, \dots, m$, the polynomial $f_j$ involves only variables $\{x_i \, | \, i \in I_j \}$.
One also suppose that the collection $\{I_1, \dots, I_m \}$ satisfies the so-called {\em running intersection property}:
\begin{definition}
The collection $\{I_1, \dots, I_m \}$ of subsets $I_j \subseteq \{1, \dots, n\}$ satisfies the running intersection property if the following holds for each $j=1 , \dots, m - 1$:
\[    I_{j + 1} \cap \bigcup_{k=1}^j I_k \subseteq I_l \quad \text{for some } l \leq j  . \]
\end{definition}
The following assumption allows to apply the sparse representation result of~\cite[Corollary 3.9]{Las06Sparse} while using either Method 1 or Method 2.
\begin{assumption}
\label{hyp:sparse}
The index set $\{1, \dots, n\}$ is partitioned into $m$ disjoint sets $I_j$, $j=1,\dots,m$ so that:
\begin{enumerate}
\item The collection $\{I_1, \dots, I_m \}$ satisfies the running intersection property.
\item For each $j=1,\dots,n^\S$, there exists some $k_j$ such that the polynomial $g_j^\S$ in (\ref{eq:defS})
 involves only variables $\{x_i \, | \, i \in I_{k_j} \}$.
\item In the definition (\ref{eq:defS}) of $\S$, we replace the inequality constraint $N^\S - \| \x \|_2^2 \geq 0$ by the $m$ quadratic constraints: 
\[N_j - \sum_{i \in I_j} x_i^2 \geq 0 , \quad j=1,\dots,m   . \]
\end{enumerate}
\end{assumption}
For each $j=1, \dots, m$, index the variable $y_j$ by $n+j$ and define $I_j^{(1)} := I_j \bigcup \{n+1, \dots, n+m\}$.
\begin{proposition}
\label{th:rip}
Under Assumption~\ref{hyp:sparse}, the collection $\{I_1^{(1)}, \dots, I_m^{(1)} \}$ of subsets $I_j^{(1)} \subseteq \{1, \dots, n, n+1, \dots, n + m\}$ satisfies the running intersection property.
\end{proposition}
\begin{proof}
\label{pr:rip1}
The collection $\{I_1, \dots, I_m \}$ of subsets $I_j \subseteq \{1, \dots, n\}$ satisfies the running intersection property. For each $j = 1, \dots, m-1$, there exists $l \leq j$ such that  $I_{j + 1} \cap \bigcup_{k=1}^j I_k \subseteq I_l$. Thus, $I_{j + 1}^{(1)} \cap \bigcup_{k=1}^j I_k^{(1)} = (I_{j + 1} \cap \bigcup_{k=1}^j I_k) \bigcup \{n+1, \dots, n+m\} \subseteq I_l \bigcup  \{n+1, \dots, n+m\} = I_l^{(1)}$, the desired result.
\end{proof}

Then Assumption~\ref{hyp:sparse} allows one to apply the sparse representation result of~\cite[Corollary 3.9]{Las06Sparse} to the semidefinite program~\eqref{eq:primalexists} associated with Method 1. Indeed, observe that the polynomial $(\x, \y) \mapsto h_f(\x, \y)$ can be decomposed as $h_f = \sum_{j=1}^m {h_f}_j$, where for each $j=1, \dots, m$, the polynomial ${h_f}_j$ involves only the variables $\{x_i \, | \, i \in I_j \}$ (the same variables involved in $f_j$) and $\y$. 

Under Assumption~\ref{hyp:sparse}, this sparse representation result can also be applied for the semidefinite program~\eqref{eq:lpstrengthlift} associated with the lifting variant of Method 2 described in Section~\ref{sec:lift2}. This is due to the fact that for each $j=1, \dots, m$, the polynomials $g^\S_{n^\S+j}$ and $g^\S_{n^\S+2 j}$  involve only the variables $\{x_i \, | \, i \in I_j \}$ and $\y$.
\section{Application examples}
\label{sec:benchs}
Here we present some application examples together with numerical results. In particular, this section illustrates that our methodology is a unified framework which can tackle important special cases: semi-algebraic set projections (Section~\ref{sec:projection})
and Pareto curves approximations (Section~\ref{sec:pareto}). Moreover, the framework can be extended to approximate images of semi-algebraic sets under semi-algebraic applications (Section~\ref{sec:saimage}).

The numerical results are given after solving either the semidefinite program~\eqref{eq:primalexists} for Method 1, the semidefinite program~\eqref{eq:lpstrength} for Method 2 or the semidefinite program~\eqref{eq:lpstrengthlift} for the lifting variant of Method 2, with the {\sc Yalmip} toolbox~\cite{YALMIP} for {\sc Matlab}.
As explained in Section \ref{sec:lift2}, the outer approximations obtained by Method 2
and its lifting variant are the same, but their semidefinite formulations differ.

Benchmarks are performed on an Intel Core i5 CPU ($2.40\, $GHz) with {\sc Yalmip} interfaced with the semidefinite programming
solver {\sc Mosek}~\cite{mosek}.

\subsection{Polynomial image of semi-algebraic sets}
\label{sec:polbenchs}

\begin{example}\label{ex:ballimage}
Consider the image of the two-dimensional unit ball $\S := \{\x \in \R^2 : \| \x \|_2^2 \leq 1 \}$ under the polynomial application $f(\x) := (x_1+x_1 x_2, x_2-x_1^3) / 2$. We choose $\B=\S$ since it can be checked that $\F=f(\S) \subset \B$.
\end{example}

 On Figure~\ref{fig:ballimageexists} resp.~\ref{fig:ballimagemeasureimage}, we represent in light gray the outer approximations $\F^1_r$ resp. $\F^2_r$ of $\F$ obtained by Method 1 resp. 2, for increasing values of
the relaxation order $r$. On each figure, the black dots correspond to the image set of the points obtained by uniform sampling of $\S$ under $f$.
We observe that the approximations behave well around the locally convex parts of the boundary of $\F$,
and  that it is not straightforward to decide whether Method 1 or Method 2 provides the best approximations. 

\begin{figure}[!ht]
\centering
\subfigure[$r=1$]{
\includegraphics[scale=\sizetinyfig]{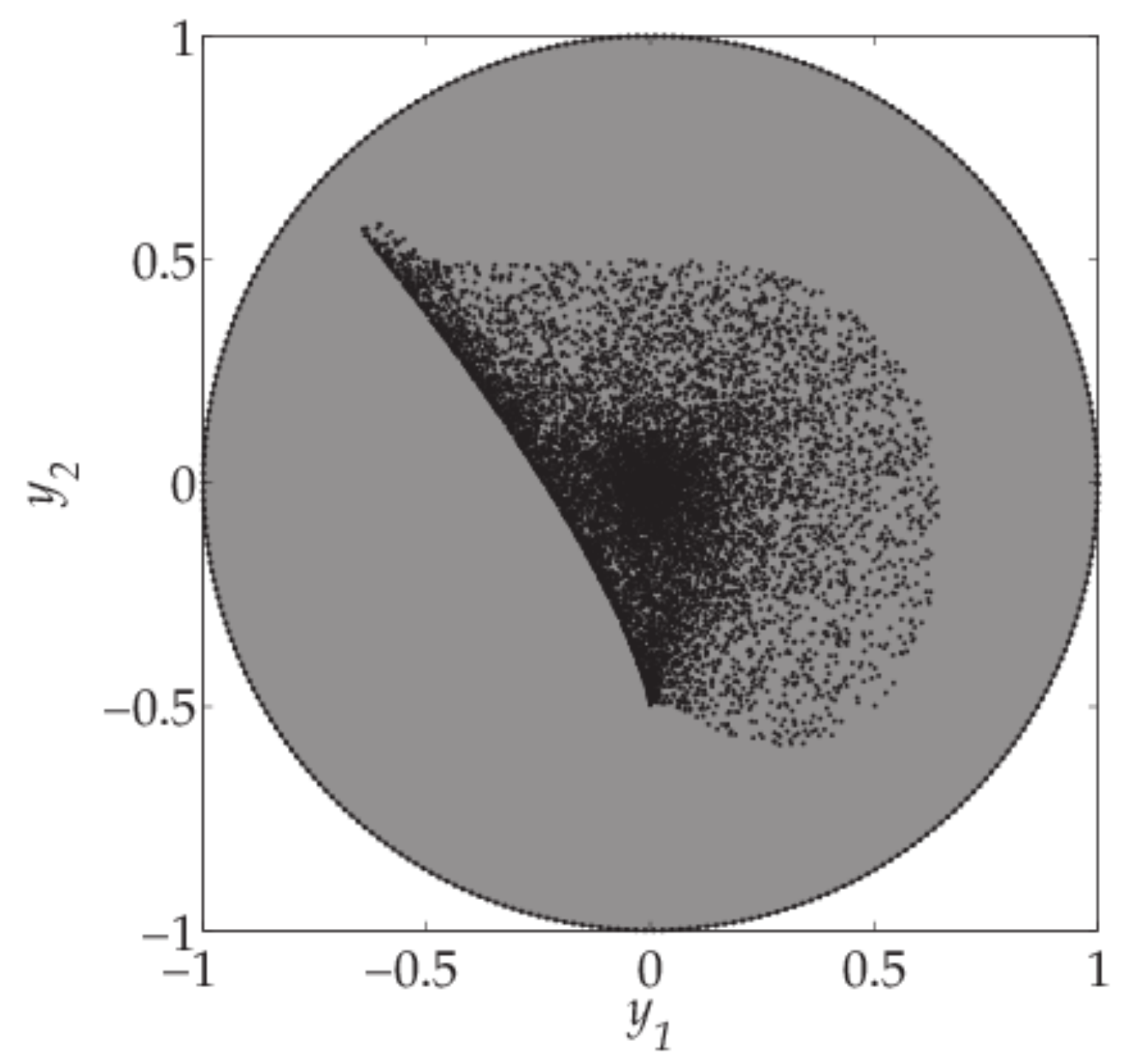}}
\subfigure[$r=2$]{
\includegraphics[scale=\sizetinyfig]{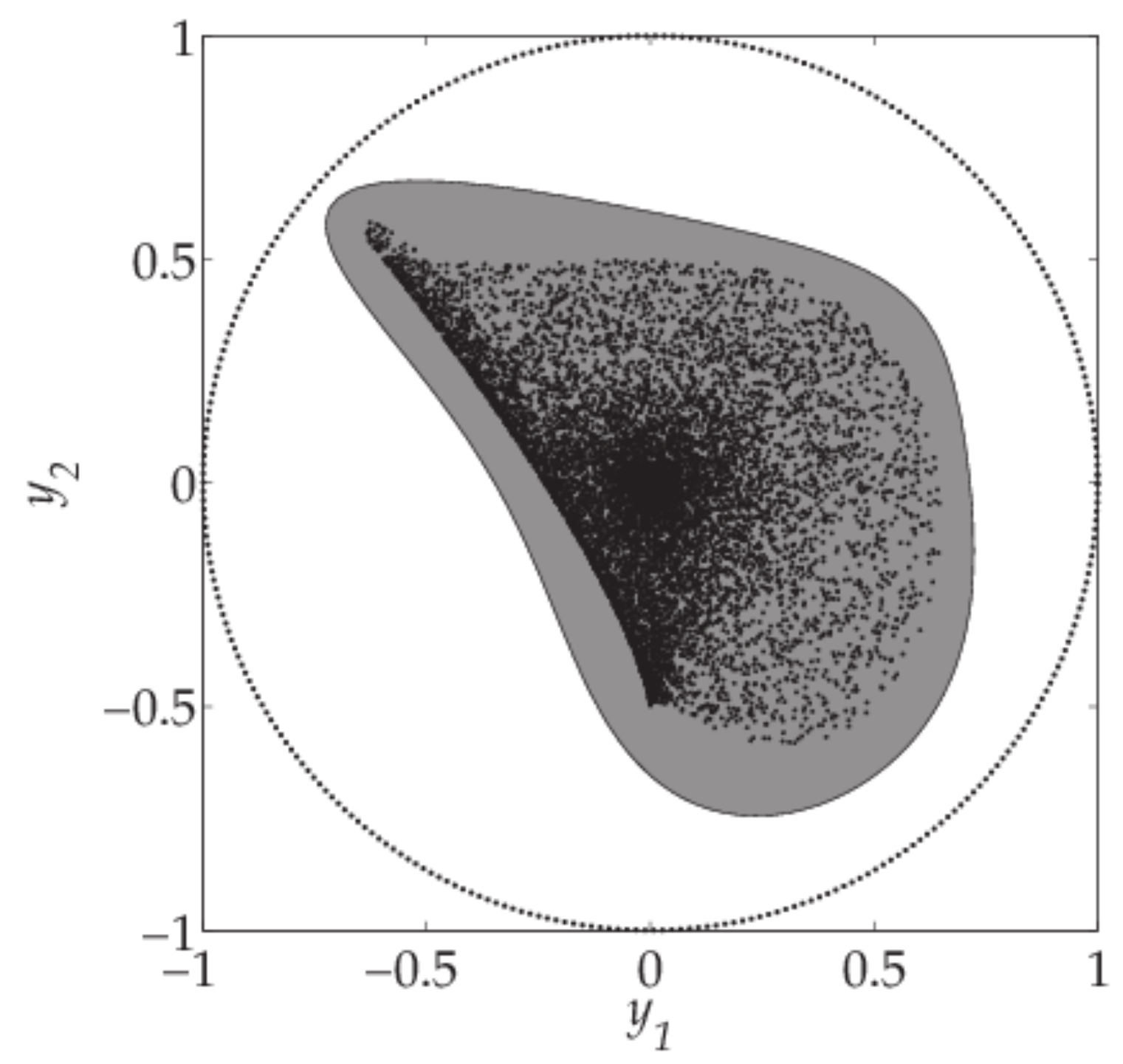}}
\subfigure[$r=3$]{
\includegraphics[scale=\sizetinyfig]{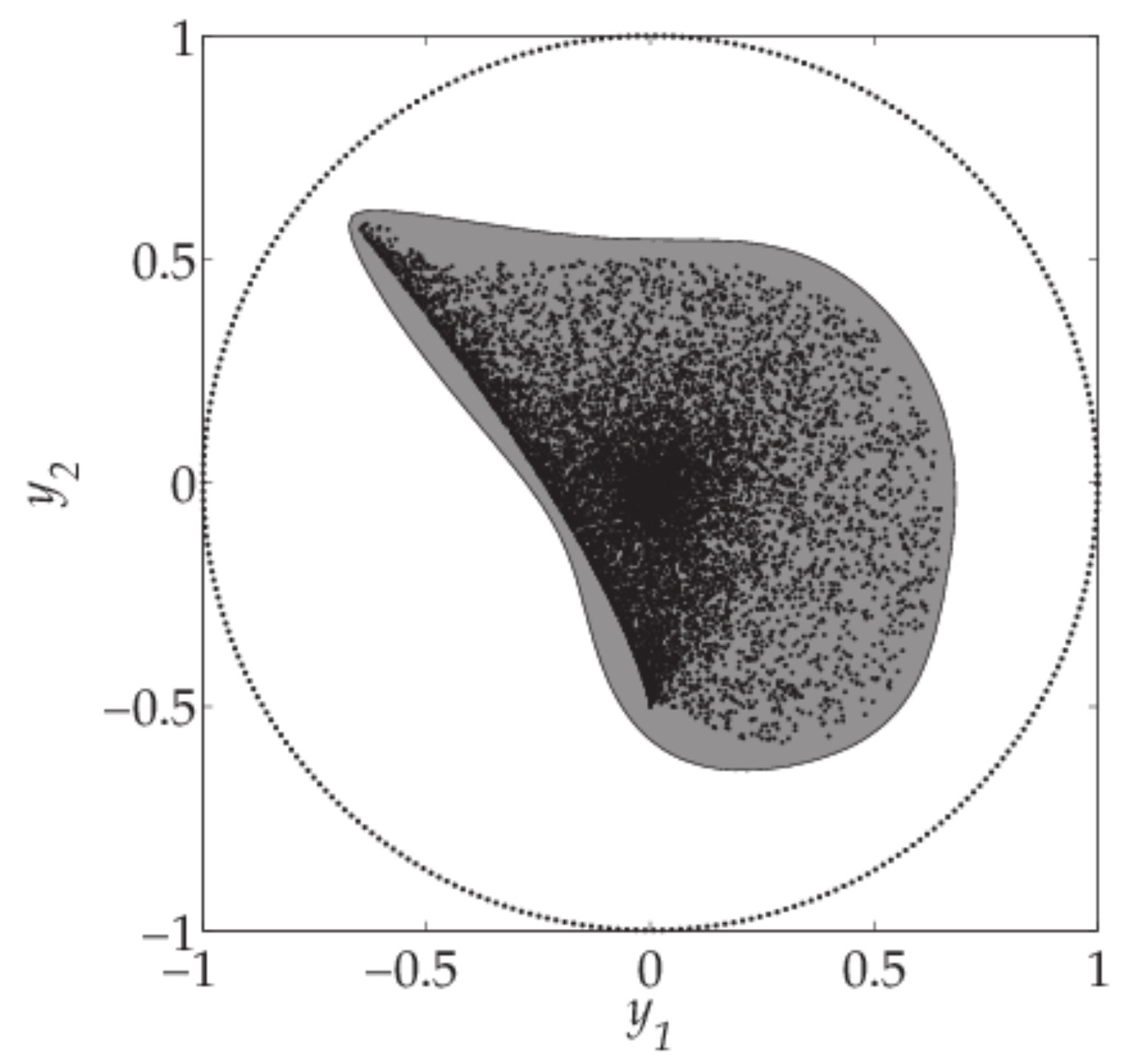}}
\subfigure[$r=4$]{
\includegraphics[scale=\sizetinyfig]{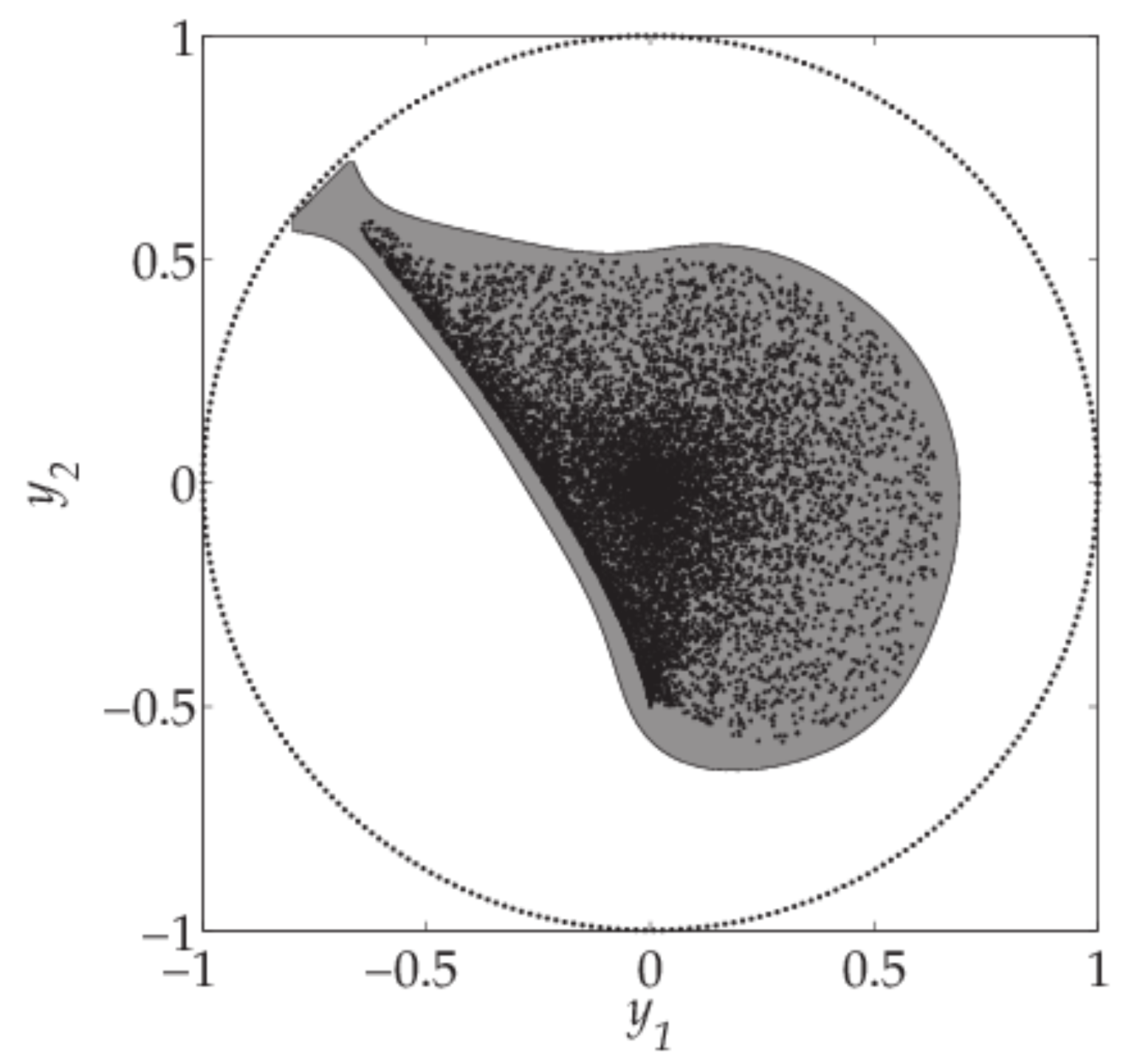}}
 \caption{Outer approximations $\F^1_r$ (light gray) of $\F$ (black dot samples)
for Example~\ref{ex:ballimage}, for $r=1,2,3,4$.}\label{fig:ballimageexists}
\end{figure}
\begin{figure}[!ht]
\centering
\subfigure[$r=1$]{
\includegraphics[scale=\sizetinyfig]{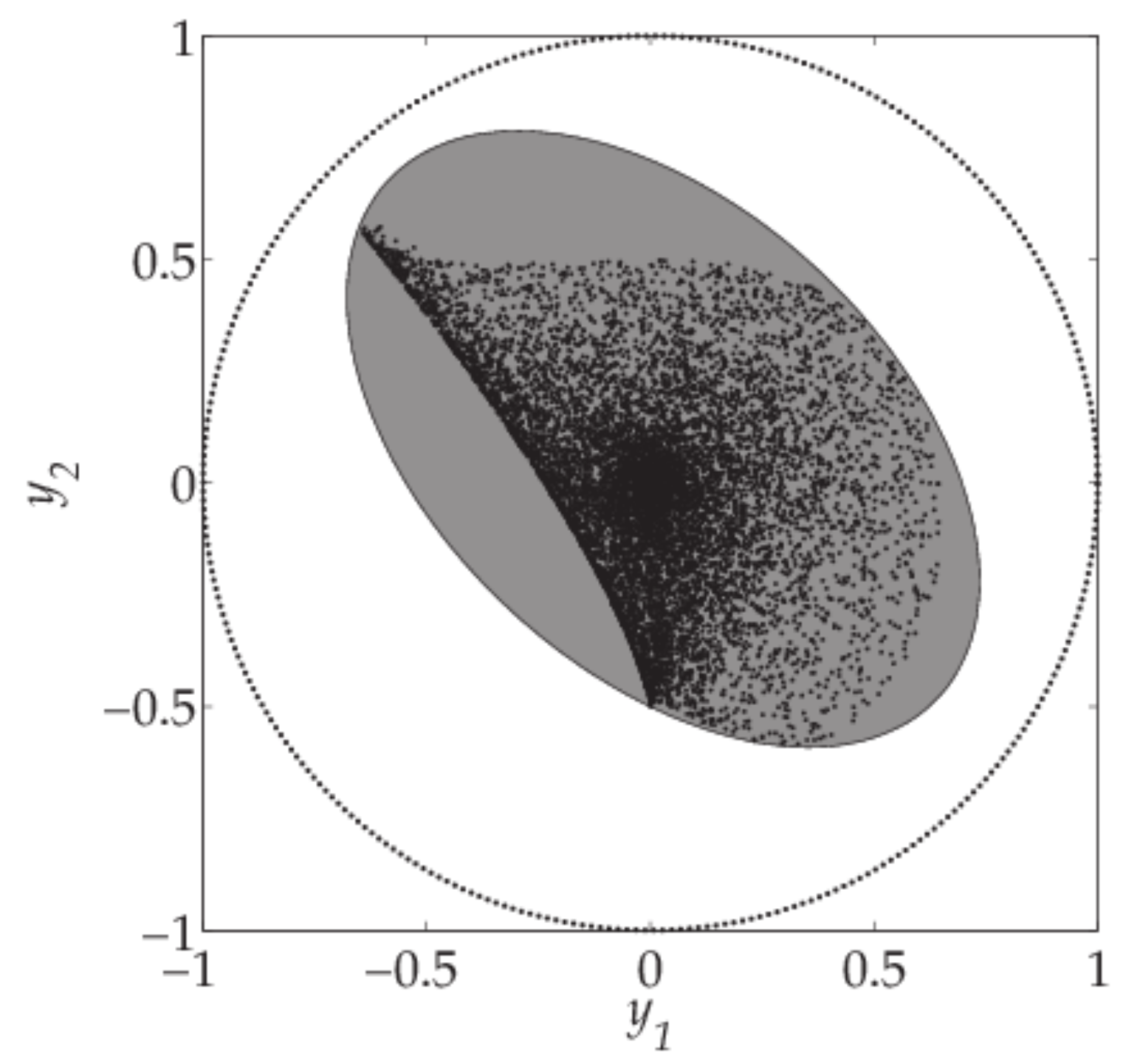}}
\subfigure[$r=2$]{
\includegraphics[scale=\sizetinyfig]{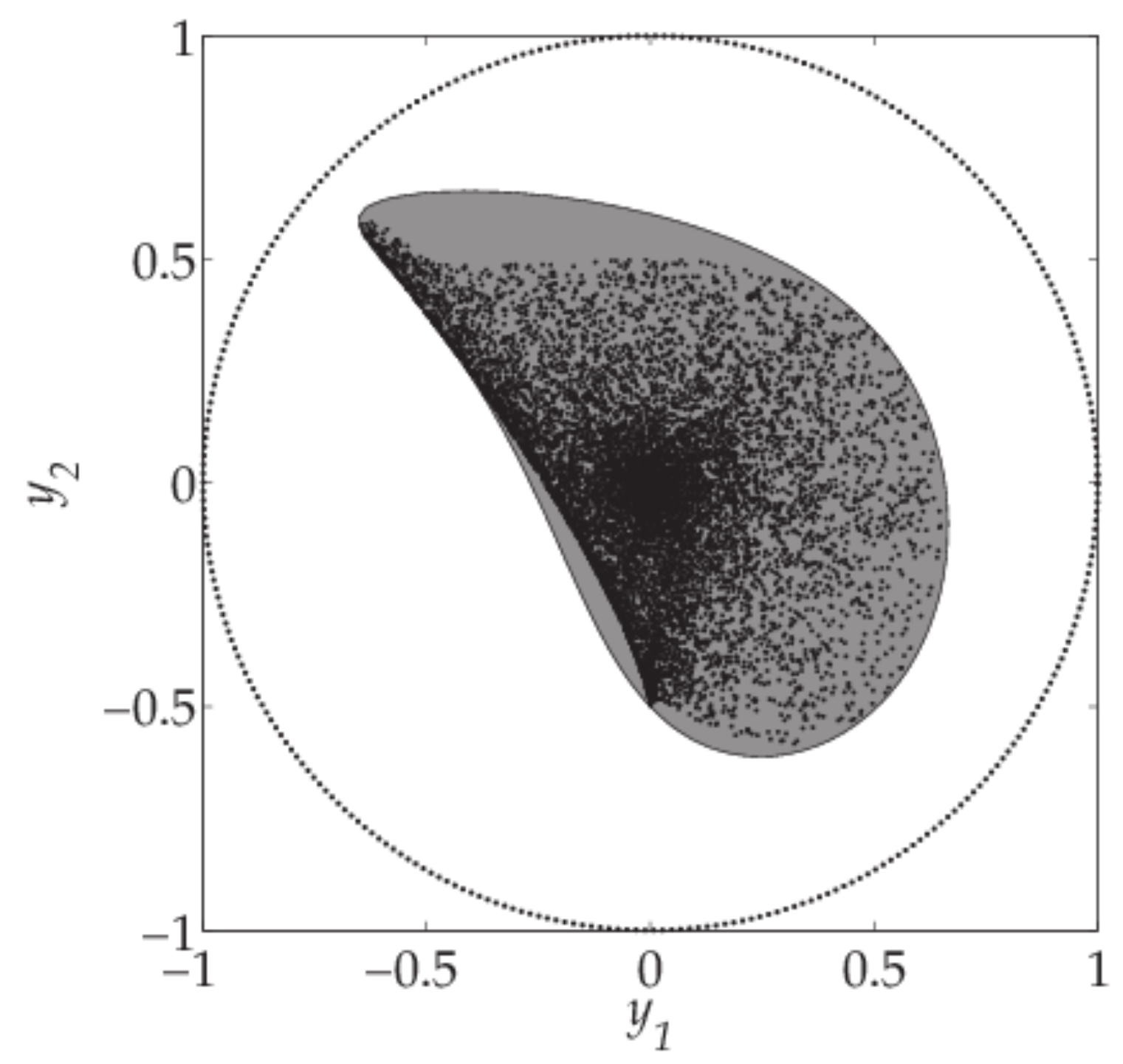}}
\subfigure[$r=3$]{
\includegraphics[scale=\sizetinyfig]{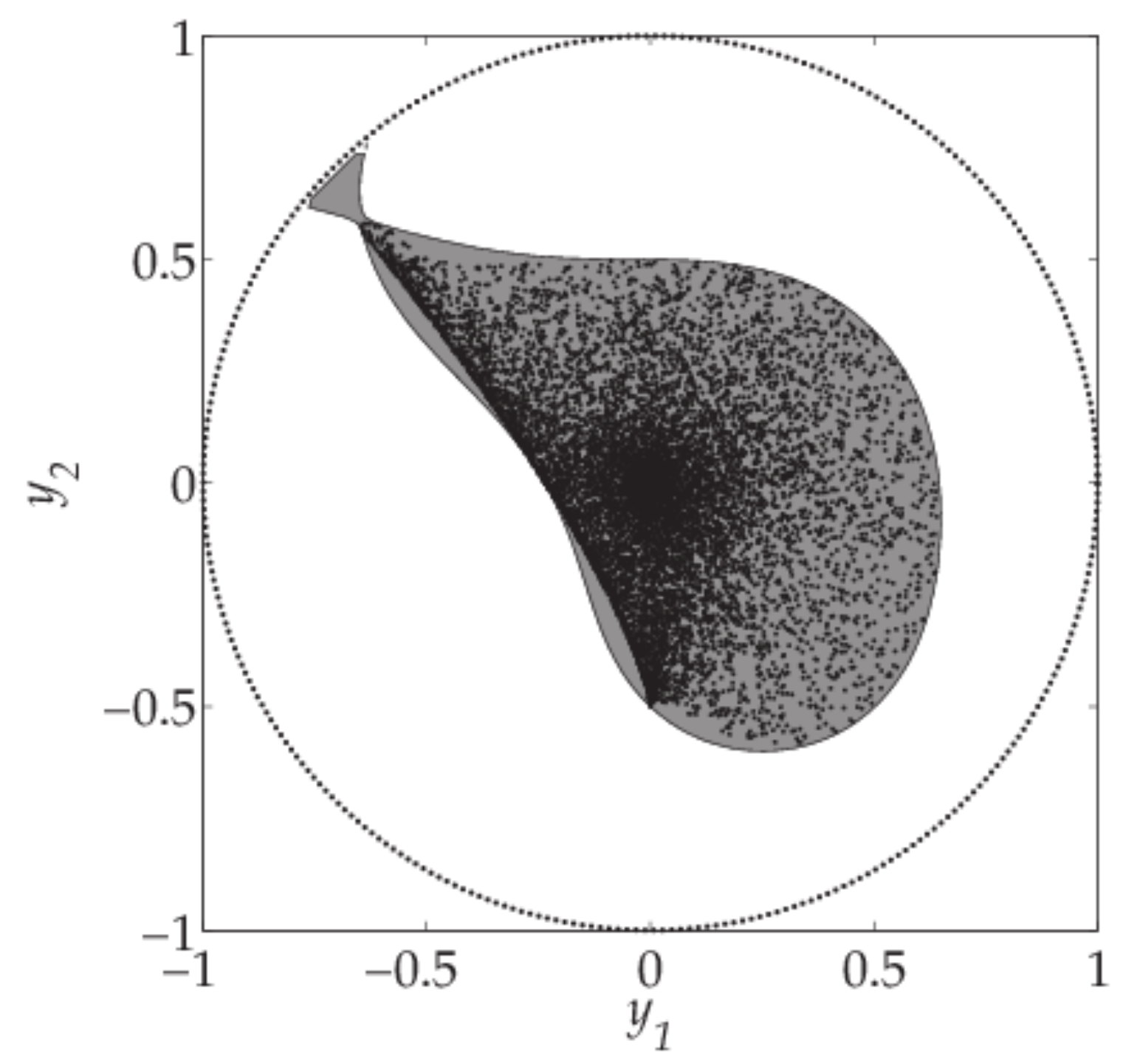}}
\subfigure[$r=4$]{
\includegraphics[scale=\sizetinyfig]{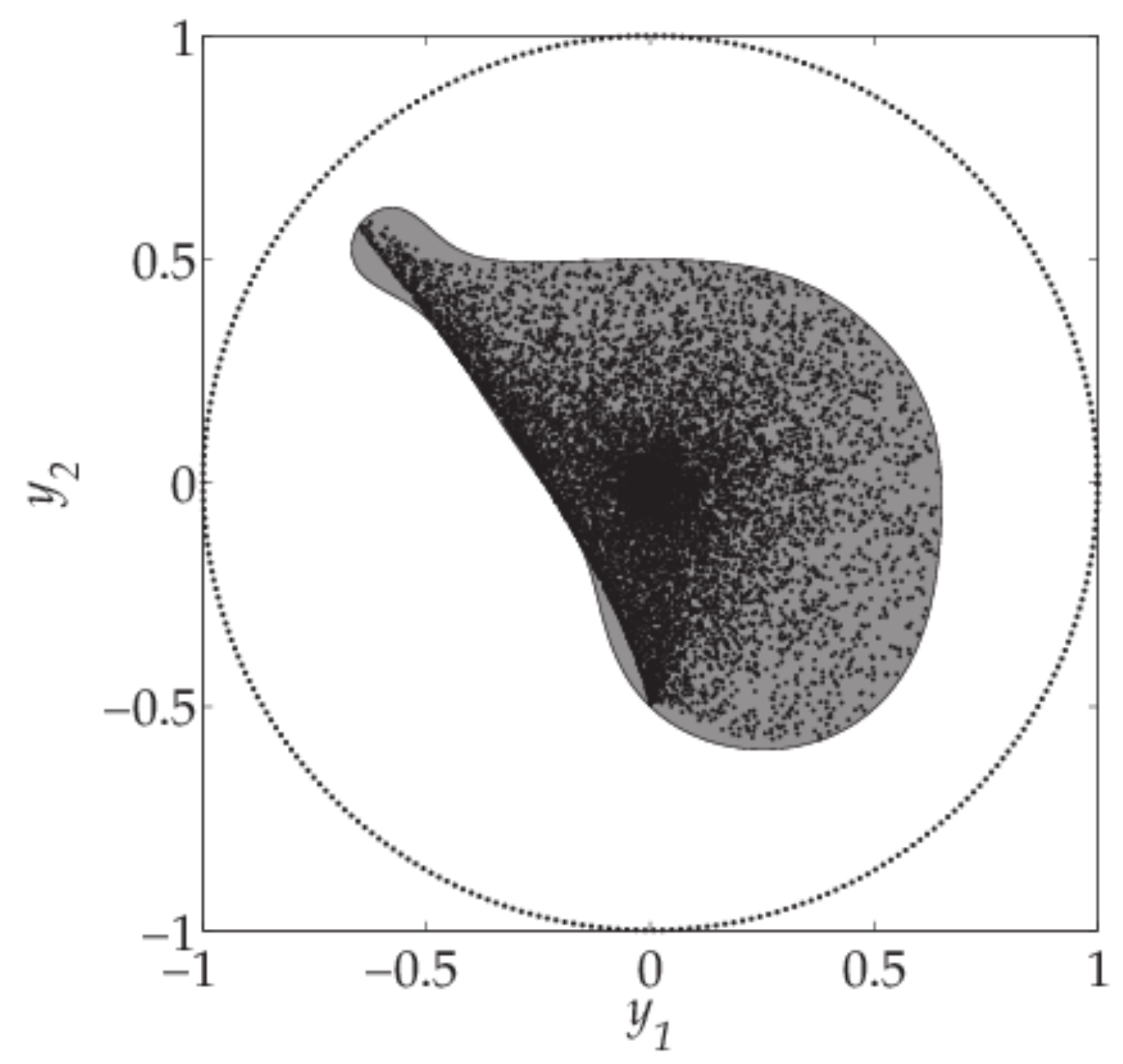}}
 \caption{Outer approximations $\F^2_r$ (light gray) of $\F$ (black dot samples)
for Example~\ref{ex:ballimage}, for $r=1,2,3,4$. }	\label{fig:ballimagemeasureimage}
\end{figure}

We indicate in Table~\ref{table:image} the data related to the semidefinite programs solved by {\sc Mosek} to compute approximations of increasing degrees, while using Method 1, Method 2 and Method 2 with the lifting strategy (see Section~\ref{sec:lift2} for more details). 
For each problem, ``vars'' stands for the total number of variables and ``size'' stands for the size of the semidefinite matrices. 
The computational timings of Method 2 with the lifting strategy are similar to those of Method 1, for $r = 1, \dots, 5$.
However, for $r = 6$ the size of the problem is significantly smaller with Method 1.

\begin{table}[!ht]
\begin{center}
\caption{Comparison of timing results for Example~\ref{ex:ballimage}}
\begin{tabular}{p{2.3cm}|c|cccccc}
\hline
\multicolumn{2}{c|}{relaxation order $r$}
 & 1 & 2 & 3 & 4 & 5 & 6
\\
\hline            
\multirow{3}{*}{Method 1} & vars &  $40$ & $212$ & $1039$ & $4211$ & $14028$ & $40251$ \\
& size & 30 & 111 & 350 & 915 & 1991 & 3822\\
& time (s) & $0.64 $ & $0.72 $ &  $0.77 $ & $1.69 $ & $8.22 $ & $40.37  $\\
\hline
\multirow{3}{*}{Method 2} & vars &  $286$ & $2140$ & $8241$ & $22720$ & $51166$ & $100626$   \\
& size & 129 & 471 & 1029 & 1803 & 2793 & 3999\\
& time (s) & $0.65 $ & $0.74 $ &  $1.54 $ & $3.4 $ & $12.89 $ & $43.74 $\\
\hline
\multirow{3}{2.3cm}{Method 2 with lifting} & vars &  $51$ & $308$ & $1499$ & $5882$ & $19546$ & $56710$   \\
    &  size & 32 & 157 & 536 & 1411 & 3128 & 6127 \\
 & time (s) & $0.58 $ & $0.66 $ &  $0.68$ & $1.93 $ & $10.07 $ & $63.88 $\\
\hline
\end{tabular}
\label{table:image}
\end{center}
\end{table}
\subsection{Projections of semi-algebraic sets}
\label{sec:projection}
For $n \geq m$, we focus on the special case of projections. Let $f$ be the projection of $\S$ with respect to the $m$ first coordinates, i.e. $f(\x) := (x_1, \dots, x_m)$. It turns out that in this case, the semidefinite program~\eqref{eq:lpstrengthlift} associated to the lifting variant of Method 2, has the following simpler formulation:
\begin{equation}
\label{eq:lpstrengthliftproj}
\begin{aligned}
\inf\limits_{w} \quad & \sum_{\beta \in \N_{2 r}^m} w_\beta  z^\B_\beta \\		
\text{s.t.} \quad & w - 1 \in \Q_{r}(\S)  , \\
\quad & w \in \Q_r(\B)  , \\
\quad & w \in \R_{2 r}[x_1, \dots, x_m]  . \\
\end{aligned}
\end{equation}
%
\begin{example}
\label{ex:proj}
Consider the projection $\F$ on the first two coordinates of the semi-algebraic set $\S := \{\x \in \R^3 : \|x\|^2_2 \leq 1, \, 1/4 - (x_1 + 1/2)^2 - x_2^2 \leq 0, \, 1/9 - (x_1 - 1/2)^4 - x_2^4 \leq 0   \}$, which belongs to $\B := \{x \in \R^2 : \|x\|^2_2 \leq 1\}$.
\end{example}
\begin{figure}[!ht]
\centering
\subfigure[$r=2$]{
\includegraphics[scale=\sizesmallfig]{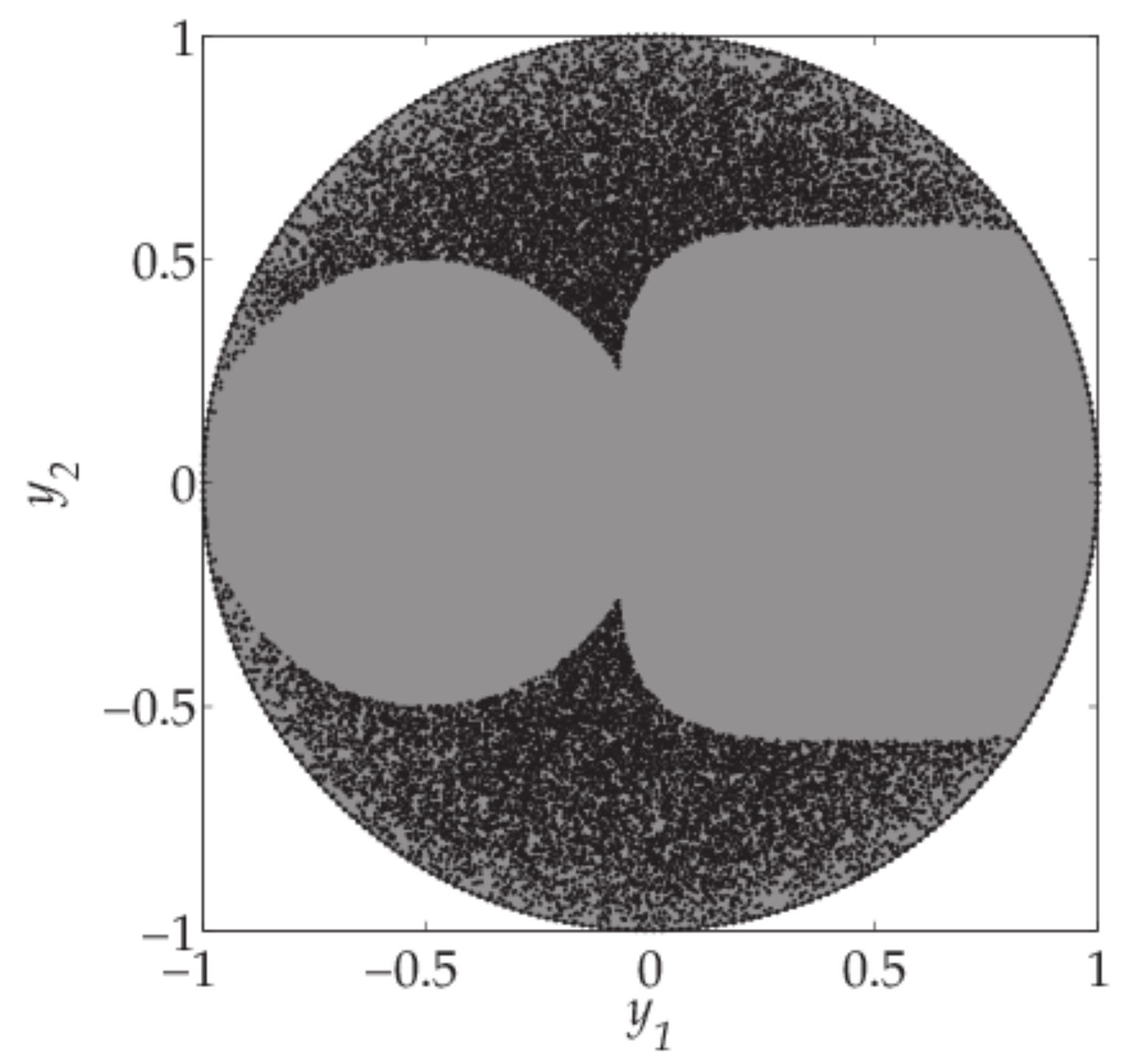}}
\subfigure[$r=3$]{
\includegraphics[scale=\sizesmallfig]{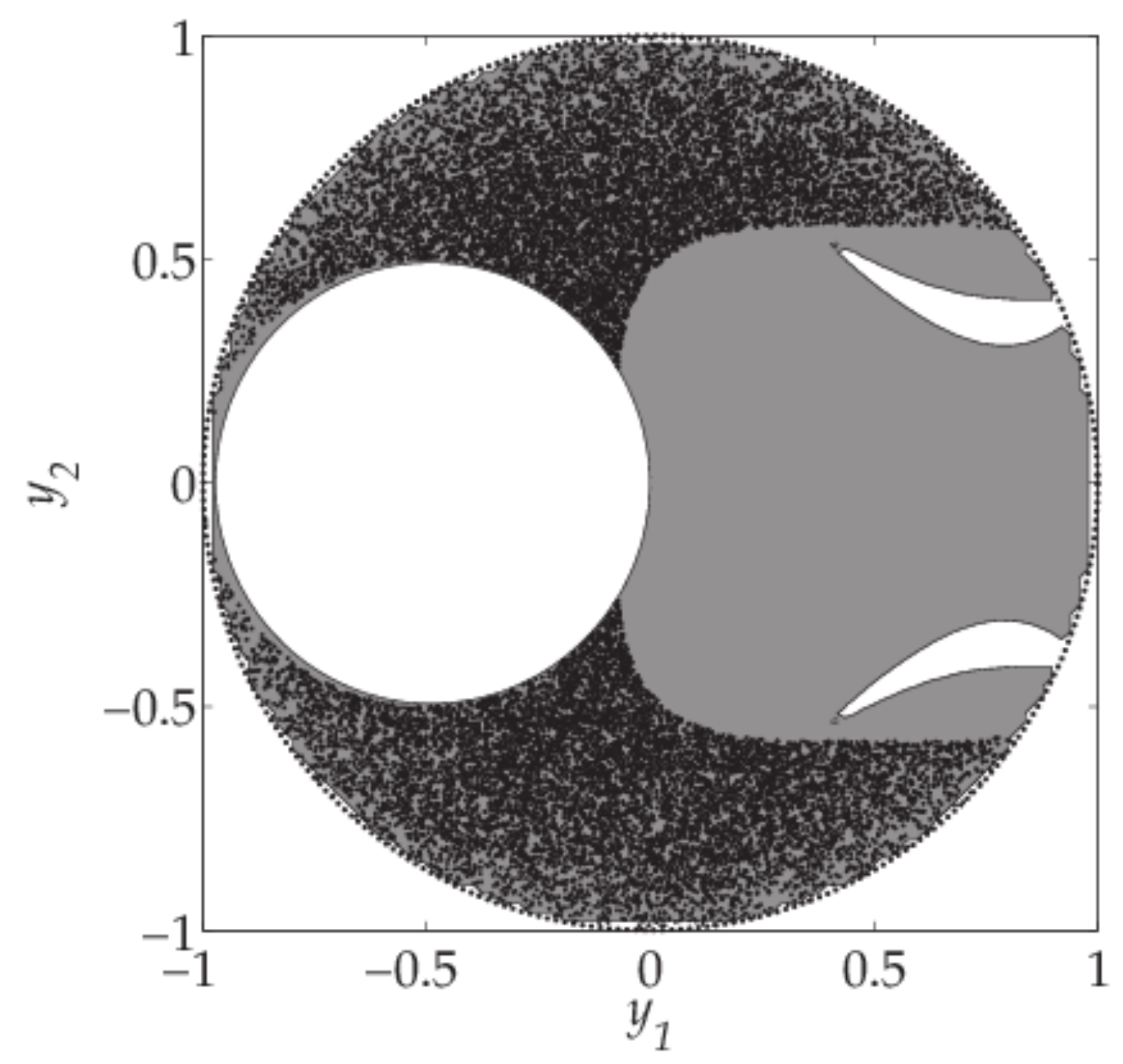}}
\subfigure[$r=4$]{
\includegraphics[scale=\sizesmallfig]{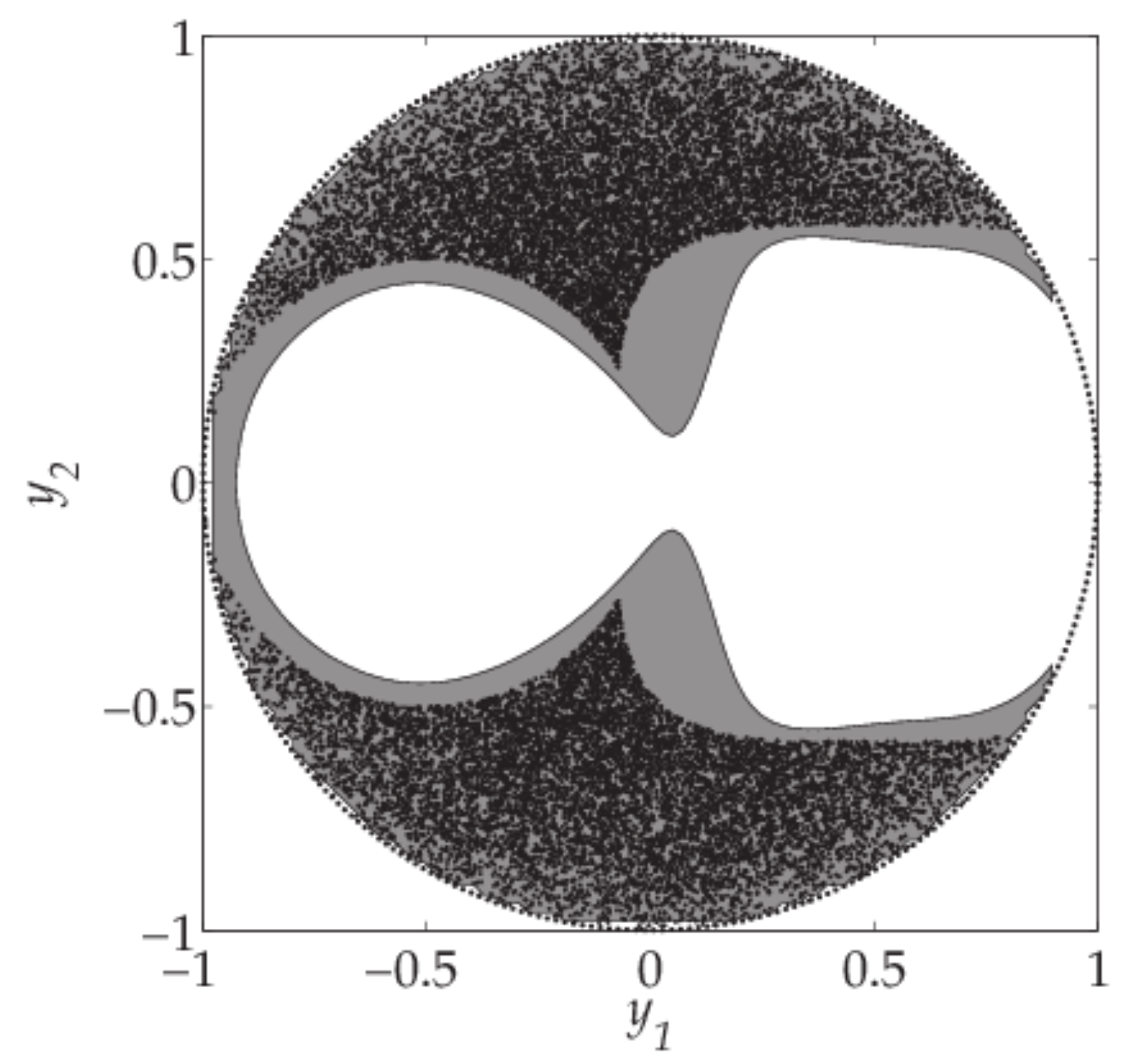}}
 \caption{Outer approximations $\F^1_r$ (light gray) of $\F$ (black dot samples)
for Example~\ref{ex:proj}, for $r=2,3,4$. }	\label{fig:projexists}
\end{figure}
\begin{figure}[!ht]
\centering
\subfigure[$r=2$]{
\includegraphics[scale=\sizesmallfig]{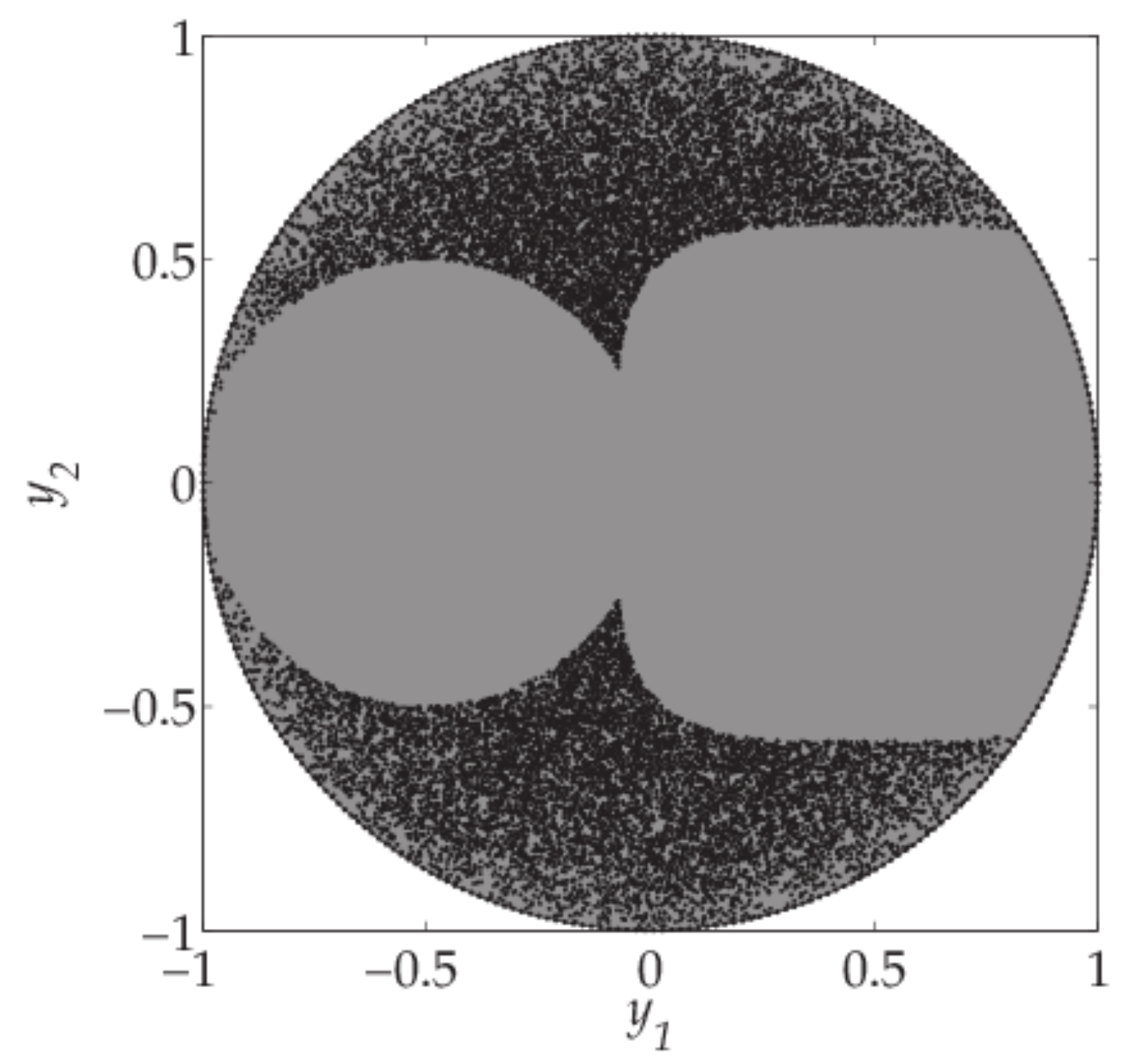}}
\subfigure[$r=3$]{
\includegraphics[scale=\sizesmallfig]{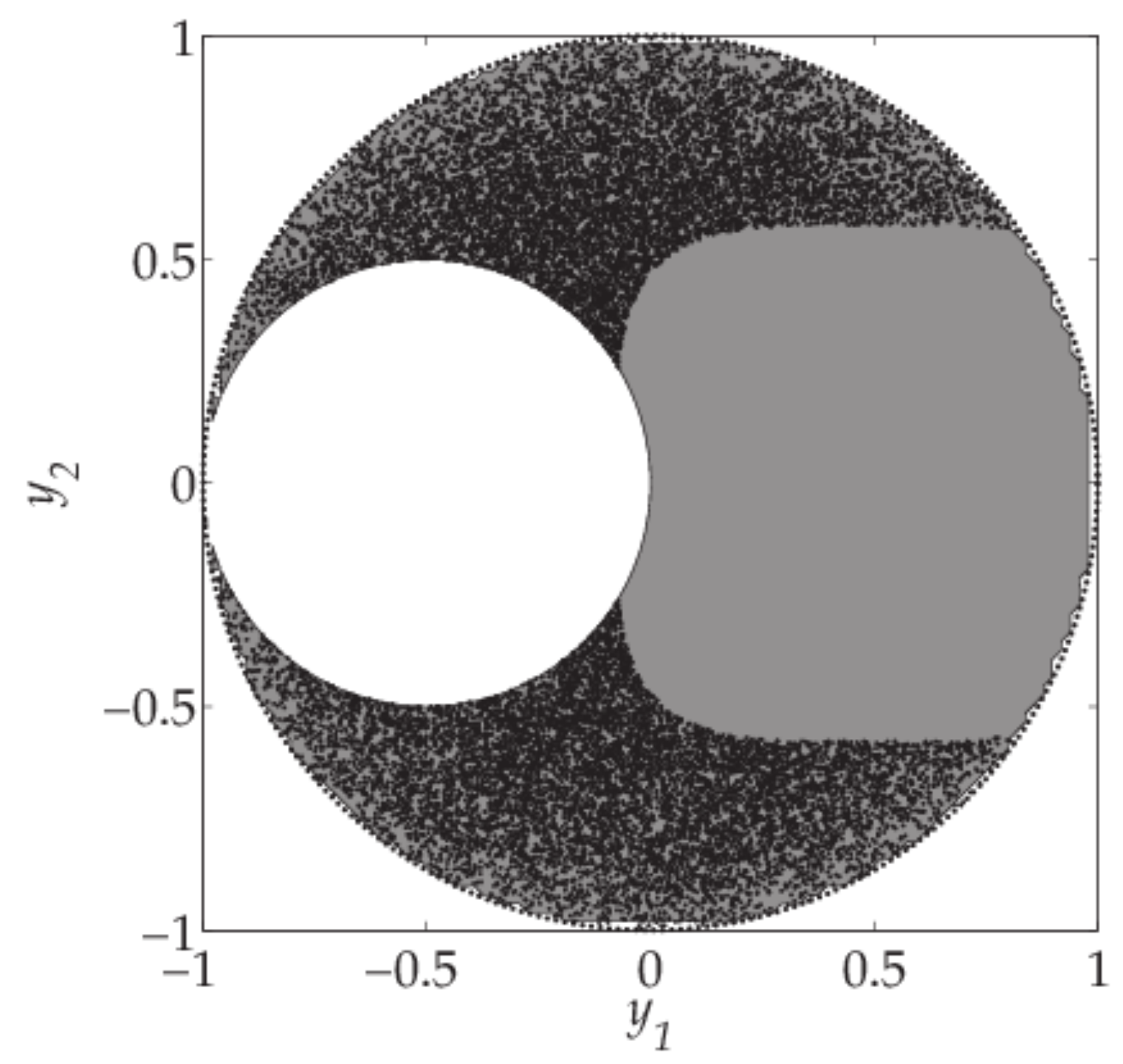}}
\subfigure[$r=4$]{
\includegraphics[scale=\sizesmallfig]{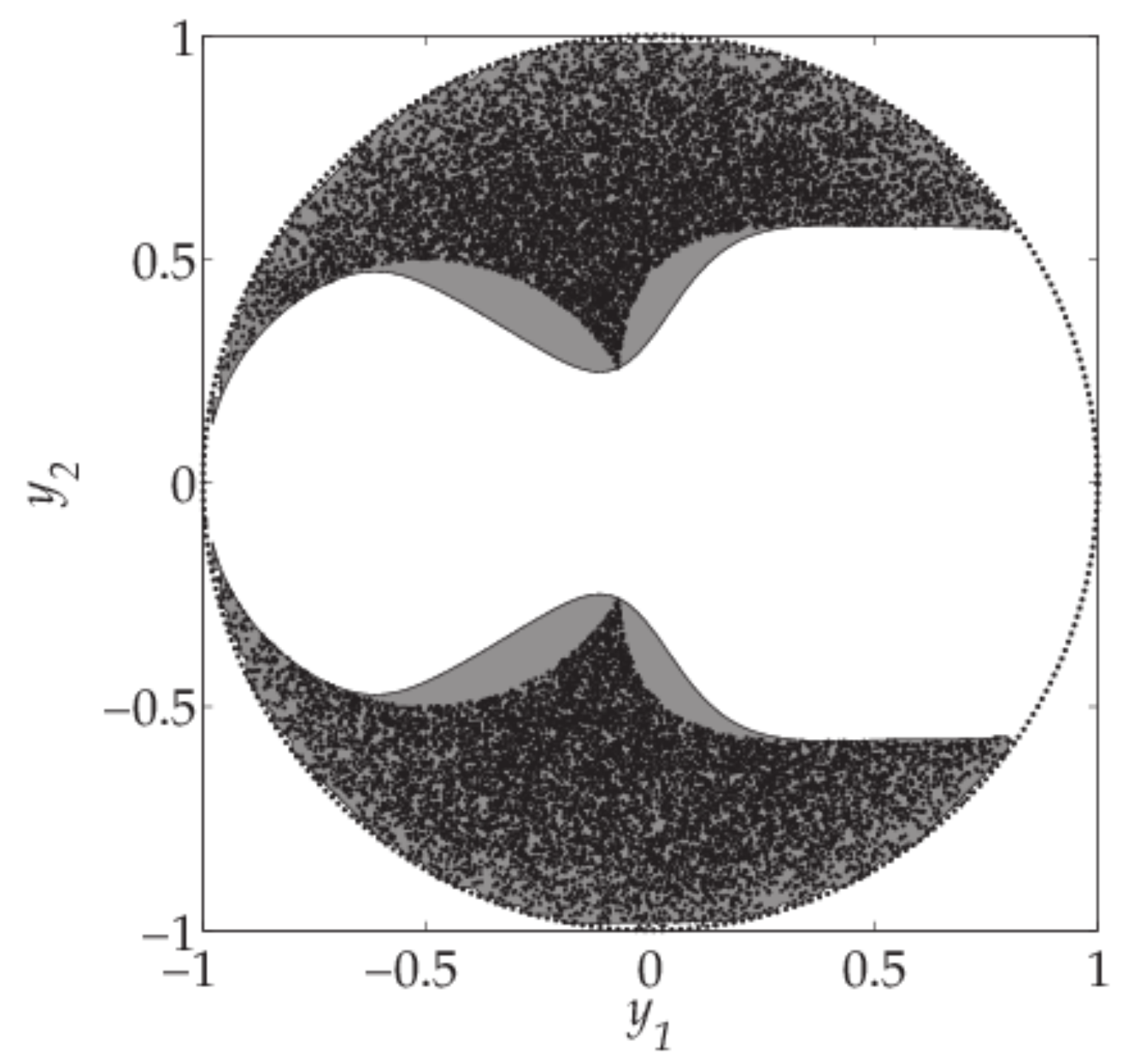}}
 \caption{Outer approximations $\F^2_r$ (light gray) of $\F$ (black dot samples)
for Example~\ref{ex:proj}, for $r=2,3,4$. }	\label{fig:projmeasureimage}
\end{figure}
Figure~\ref{fig:projexists} resp.~\ref{fig:projmeasureimage} displays approximation of the projection of $\S$ on the first two coordinates with Method 1 resp. 2.

\subsection{Approximating Pareto curves}
\label{sec:pareto}
In~\cite{orl14}, we propose a method to approximate {\em Pareto curves} associated with bicriteria polynomial optimization problems $\min_{\x \in \S}\{ (f_1(\x), f_2(\x)) \}$. 
The image space $\R^2$ is partially ordered with the positive orthant $\R_{+}^2$, that is, for every $\y_1, \y_2 \in \R^2$, $\y_1 \geq \y_2$ stands for $\y_2-\y_1 \in \R_+^2$. A point $\bar{\x}$ is  called a {\it weakly Edgeworth-Pareto optimal point}, when there is no $\x \in \S$ such that $f_j(\x) < f_j (\bar{\x}), \: j = 1,2$. The Pareto curve is the set of weakly Edgeworth-Pareto optimal points. For more details on multicriteria optimization, we refer the interested reader to~\cite{jahn:2010:vector} and the references therein.

The methodology of~\cite{orl14} consists of reformulating the initial bicriteria optimization problem to use a hierarchy of semidefinite approximations for parametric polynomial optimization problems. Then, one can apply the framework developed in~\cite{Lasserre:2010:JMA} and build a hierarchy of semidefinite programs, allowing to approximate as closely as desired the Pareto curve. Here we propose to study outer approximations of the set $\F=(f_1(\S), f_2(\S))$ since points along the boundary
of a tight outer approximation are expected to be close to the Pareto curve.

\begin{example}
\label{ex:pareto}
Let consider the two-dimensional nonlinear problem proposed in~\cite{Wilson2001}:\\
$\min_{\x \in \S}\{ (f_1(\x), f_2(\x))  \}$, with $f_1(\x):= \tfrac{(x_1+x_2-7.5)^2}{4} + (x_2-x_1+3)^2$, $f_2(\x) := x_1 + x_2^2$ and $\S := \{\x \in \R^2 : -(x_1-2)^3/2 - x_2 + 2.5 \geq 0, -x_1 - x_2 + 8 (x_2-x_1+0.65)^2 + 3.85 \geq 0 \}$. Instead of $f_1$, we consider $\tilde{f}_1 := (f_1 (\x) - a_1) / (b_1 - a_1)$, where $a_1$ and $b_1$ are given by $a_1 := \min_{\x  \in \S} f_1(\x)$ and  $b_1 := f_1(\overline{\x})$
with $\overline{\x}$ a solution of $\min_{\x \in \S} f_2(\x)$. Similarly, we consider a scaled criterion $\tilde{f}_2$ defined from $f_2$.
A preprocessing step consists in computing lower and upper bounds of the polynomial $f_1$ (resp. $f_2$) over $\S$ to define $\tilde{f} = (\tilde{f}_1, \tilde{f}_2)$. Doing so, one ensures that $\tilde{f}(\S)$ is a subset of the unit ball $\B$ and the present methodology applies.
\end{example}
\begin{figure}[!ht]
\centering
\subfigure[$r=1$]{
\includegraphics[scale=\sizesmallfig]{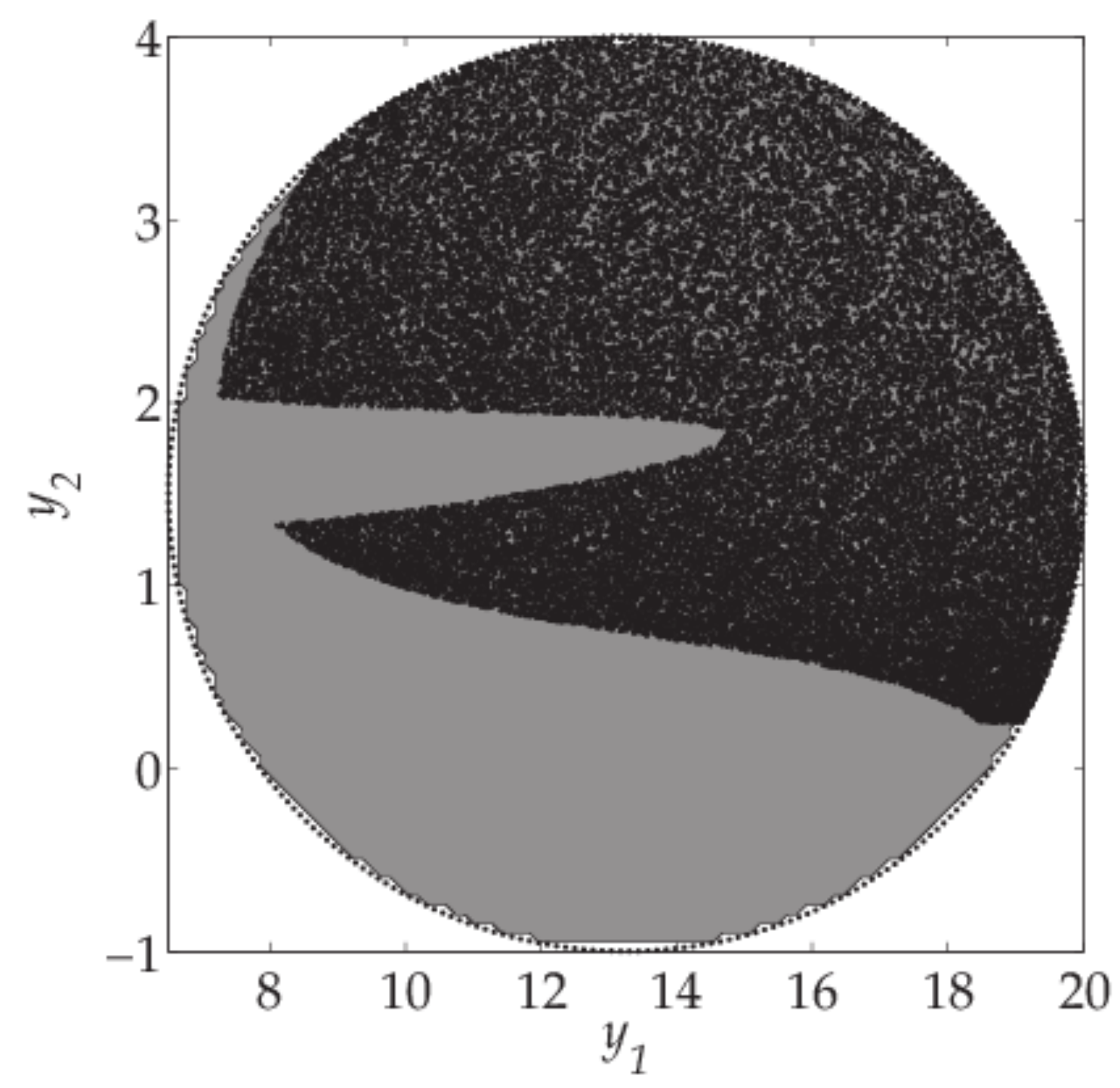}}
\subfigure[$r=2$]{
\includegraphics[scale=\sizesmallfig]{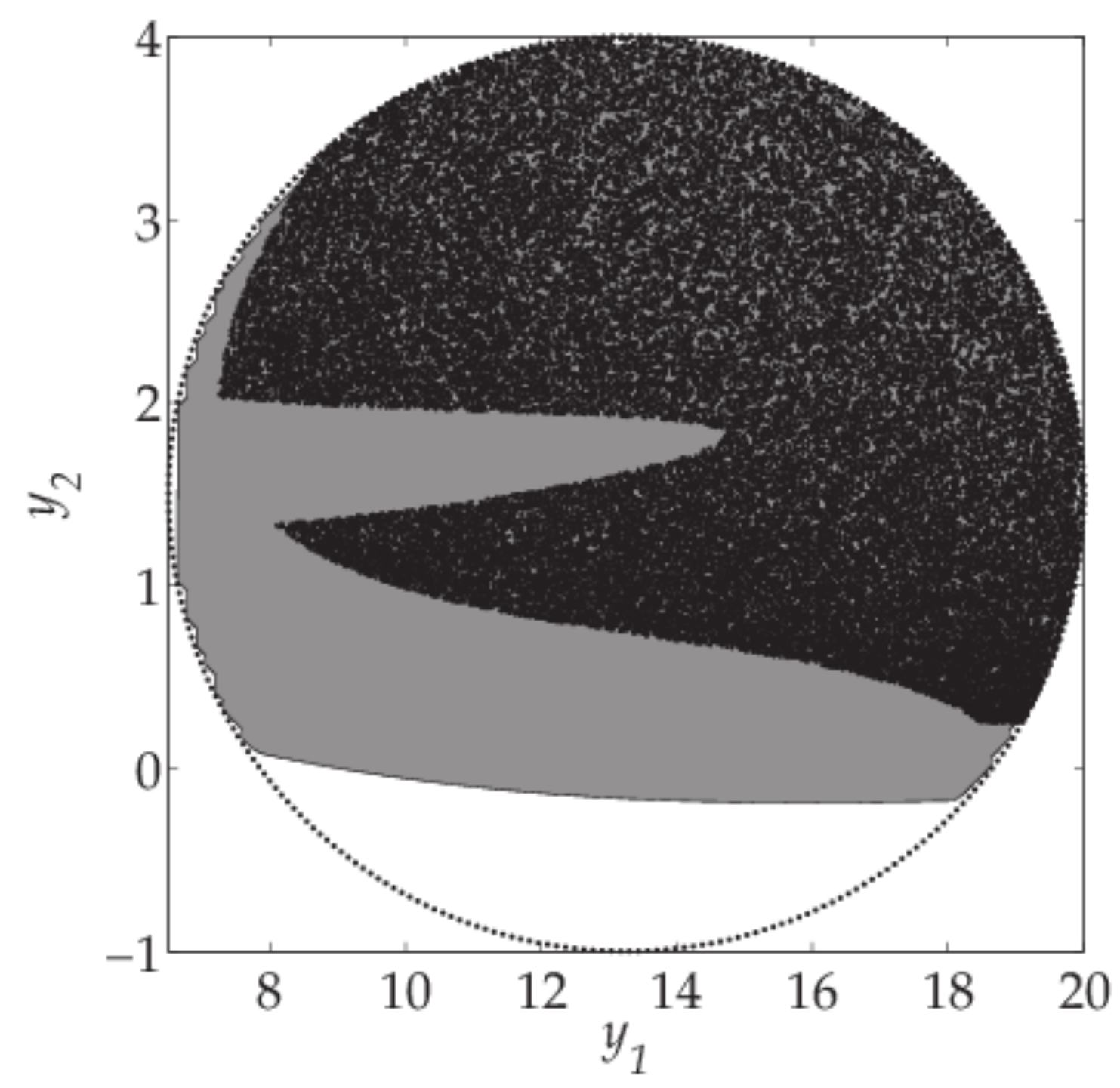}}
\subfigure[$r=4$]{
\includegraphics[scale=\sizesmallfig]{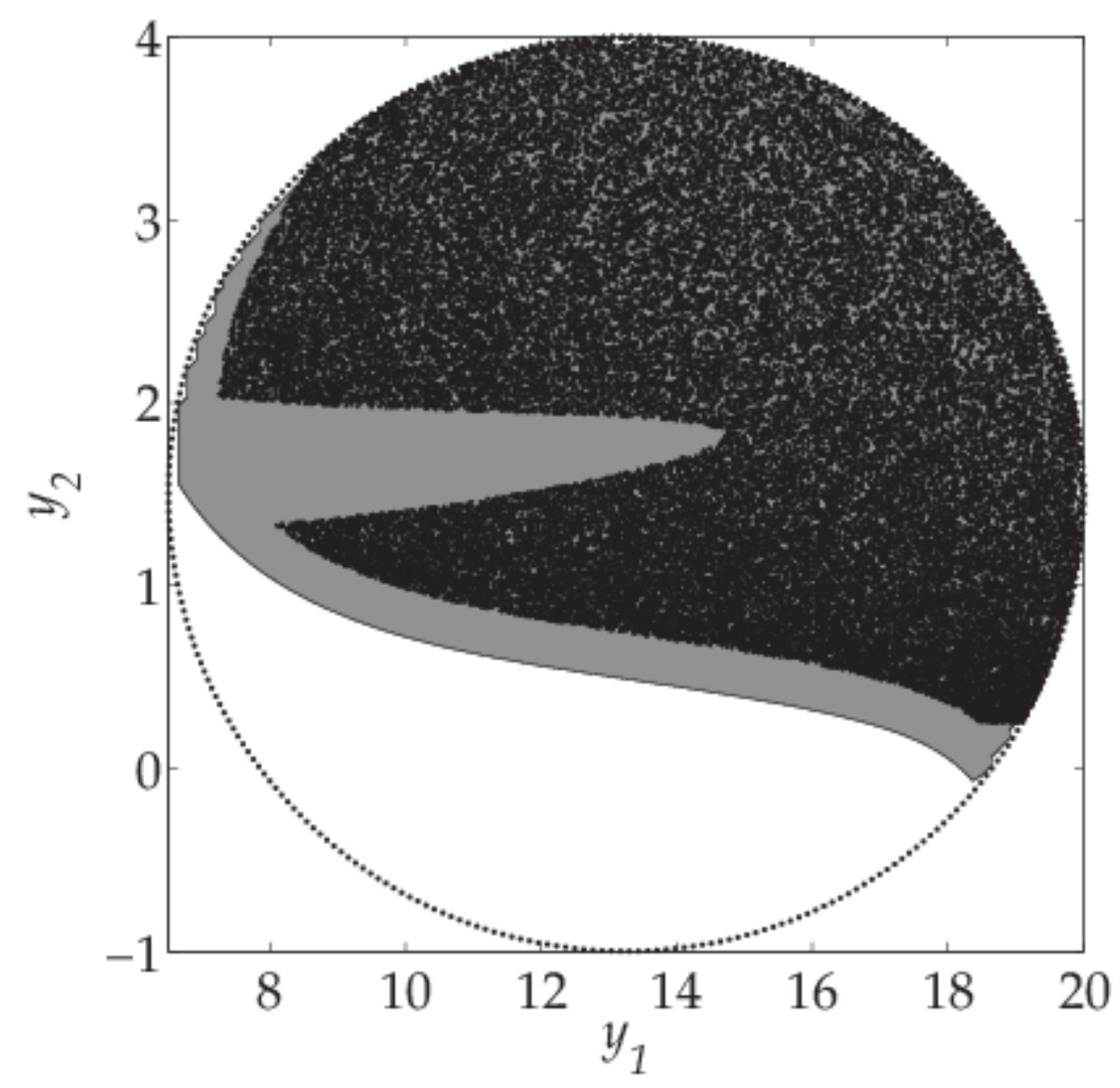}}
 \caption{Outer approximations $\F^1_r$ (light gray) of $\F$ (black dot samples)
for Example~\ref{ex:pareto}, for $r=1,2,4$.  }	\label{fig:paretoexists}
\end{figure} 
\begin{figure}[!ht]
\centering
\subfigure[$r=1$]{
\includegraphics[scale=\sizesmallfig]{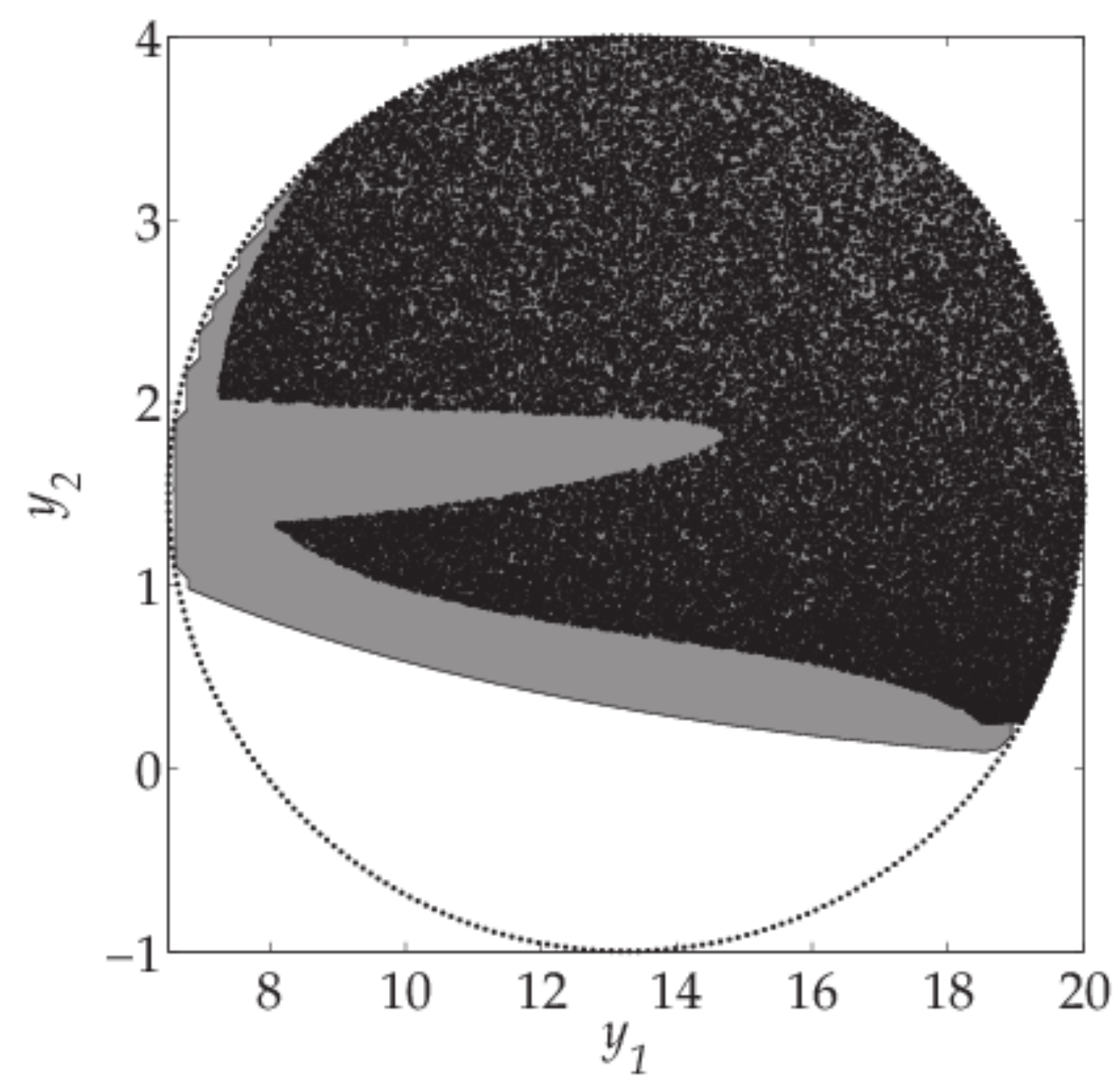}}
\subfigure[$r=2$]{
\includegraphics[scale=\sizesmallfig]{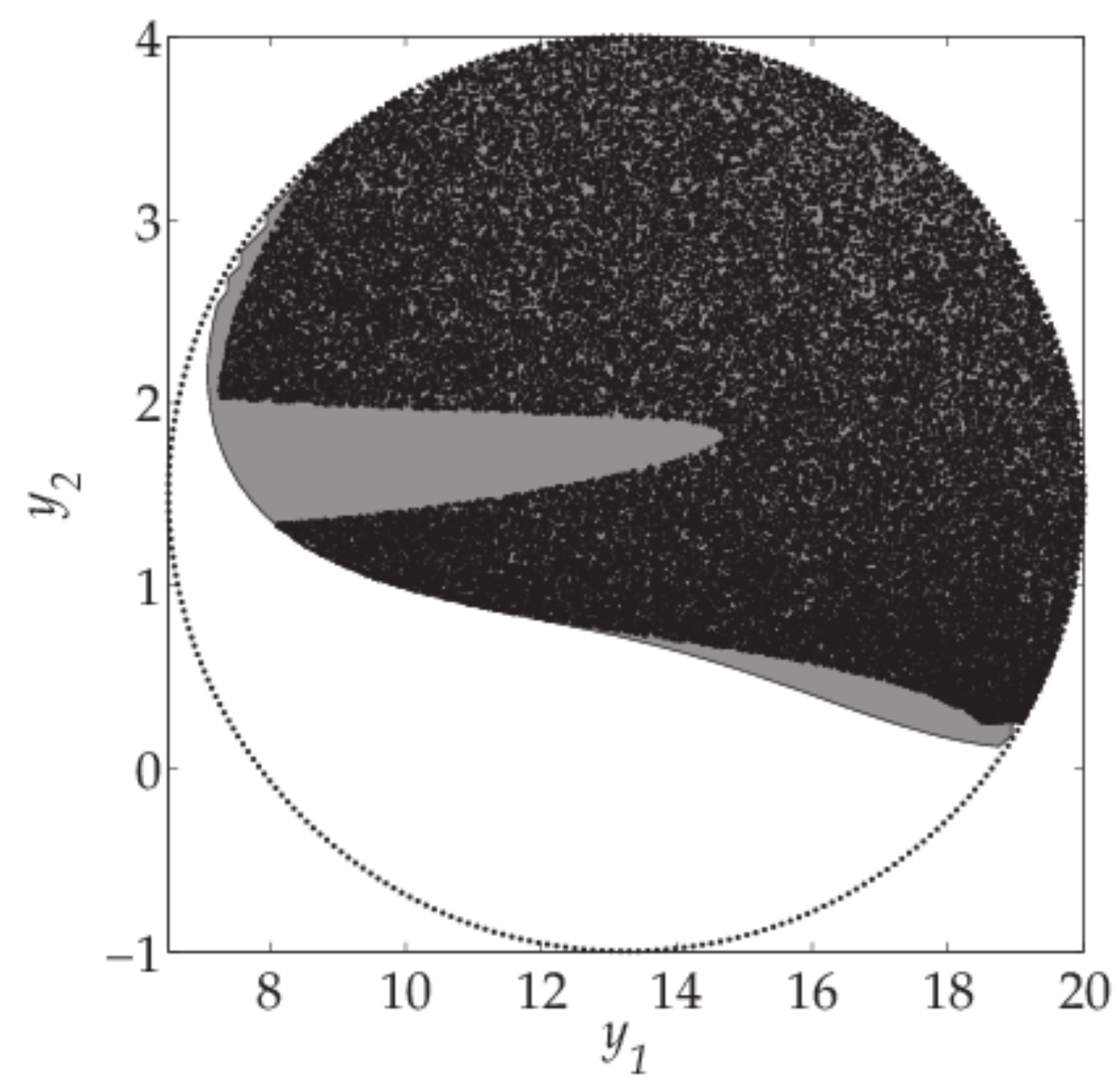}}
\subfigure[$r=4$]{
\includegraphics[scale=\sizesmallfig]{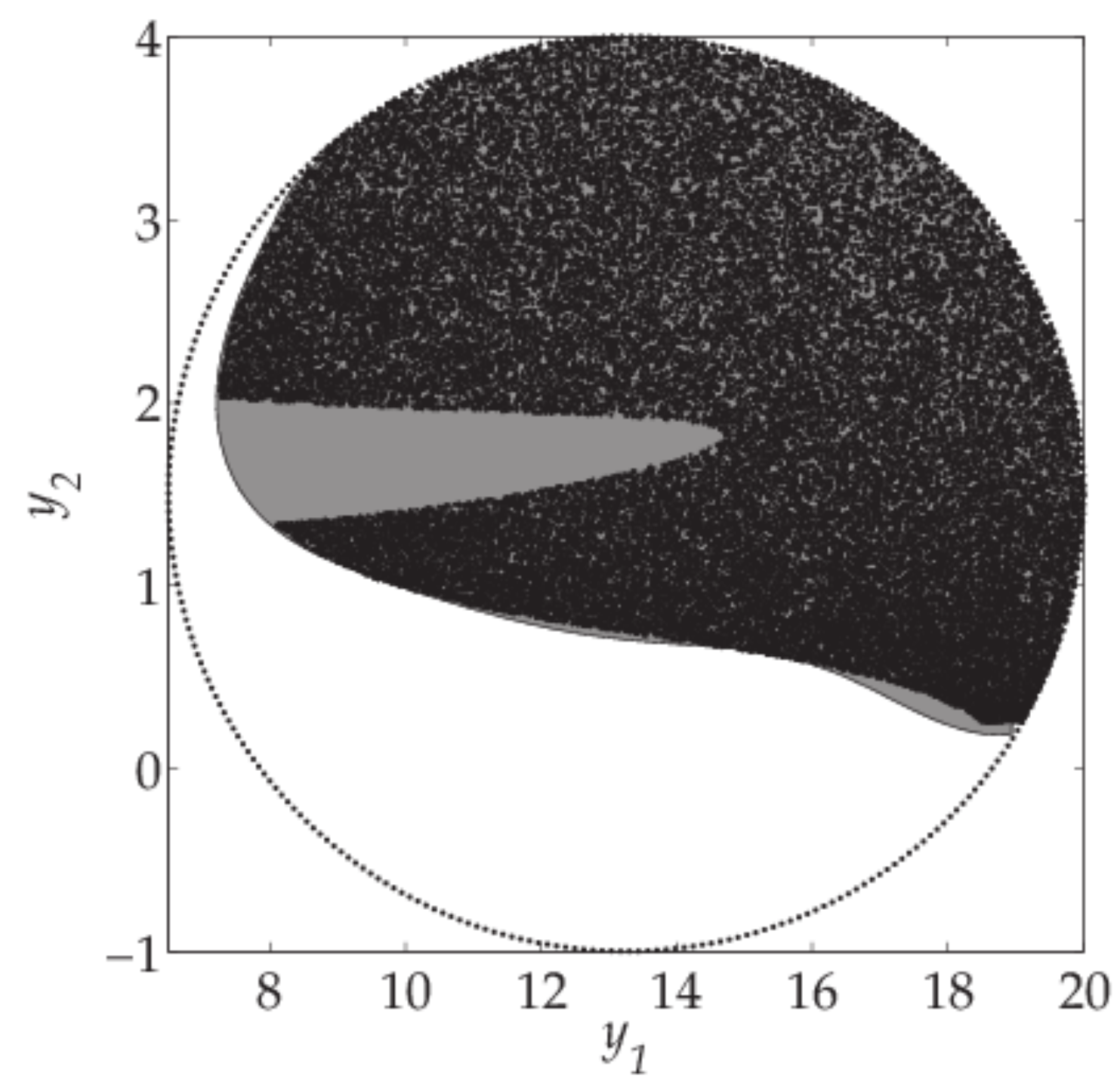}}
 \caption{Outer approximations $\F^2_r$ (light gray) of $\F$ (black dot samples)
for Example~\ref{ex:pareto}, for $r=1,2,4$.   }	\label{fig:paretomeasureimage}
\end{figure}
In this case, the Pareto curve is nonconvex and disconnected. As depicted in Figure~\ref{fig:paretoexists} and Figure~\ref{fig:paretomeasureimage}, it is difficult to obtain precise approximations of the whole image set, in particular for the subset $\F \cap \B_0$, with an ellipse $\B_0:=\{\x \in \R^2 \: :\: ((x_1-13.7)/1.7)^2+((x_2-1.8)/0.5)^2 \leq 1\}$.
Figure~\ref{fig:zoomexists} displays more precise outer approximations of degree 8 (a) and degree 10 (b).
\begin{figure}[!ht]
\centering
\subfigure[$r=4$]{
\includegraphics[scale=\sizesmallfig]{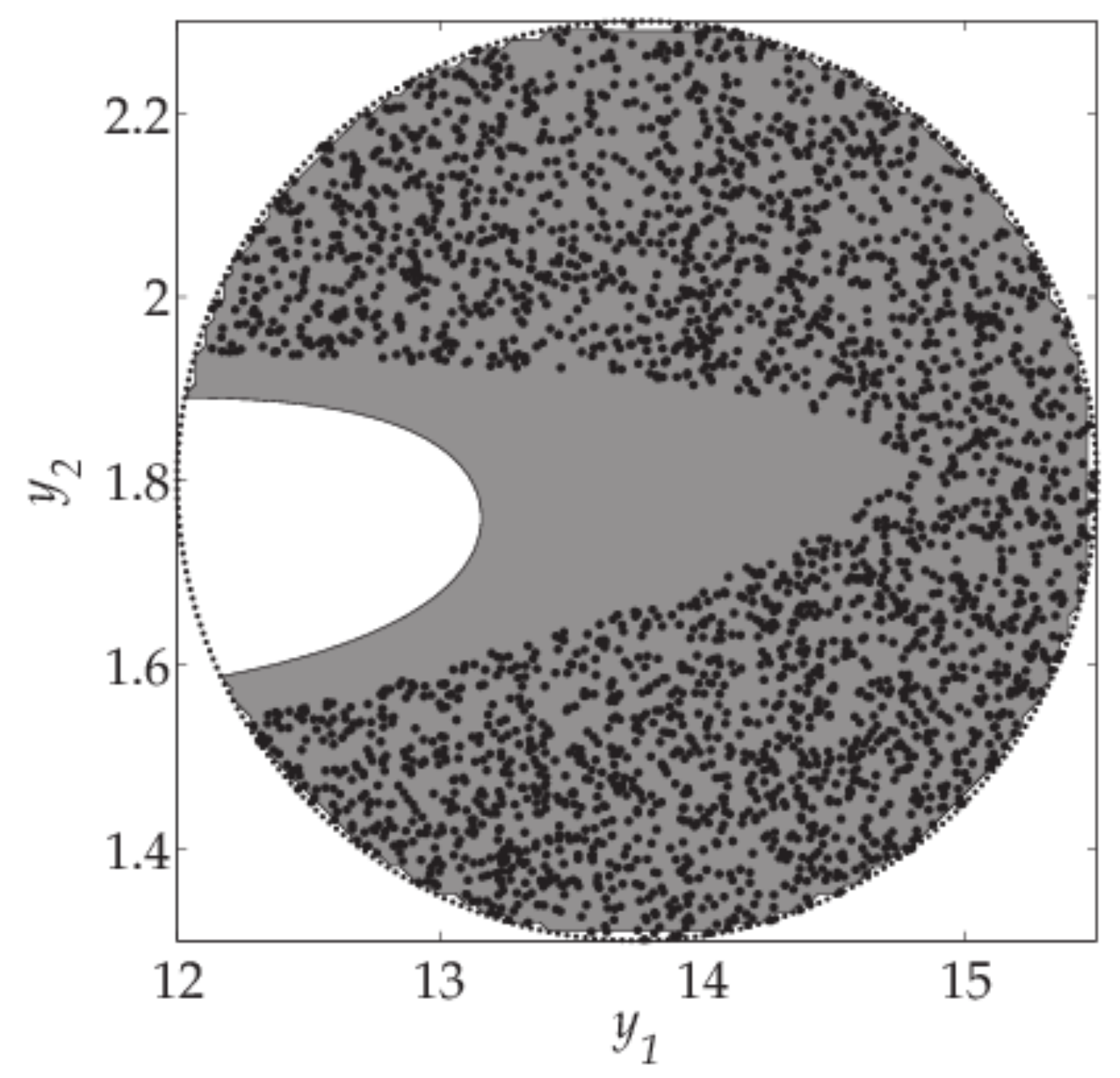}}
\subfigure[$r=5$]{
\includegraphics[scale=\sizesmallfig]{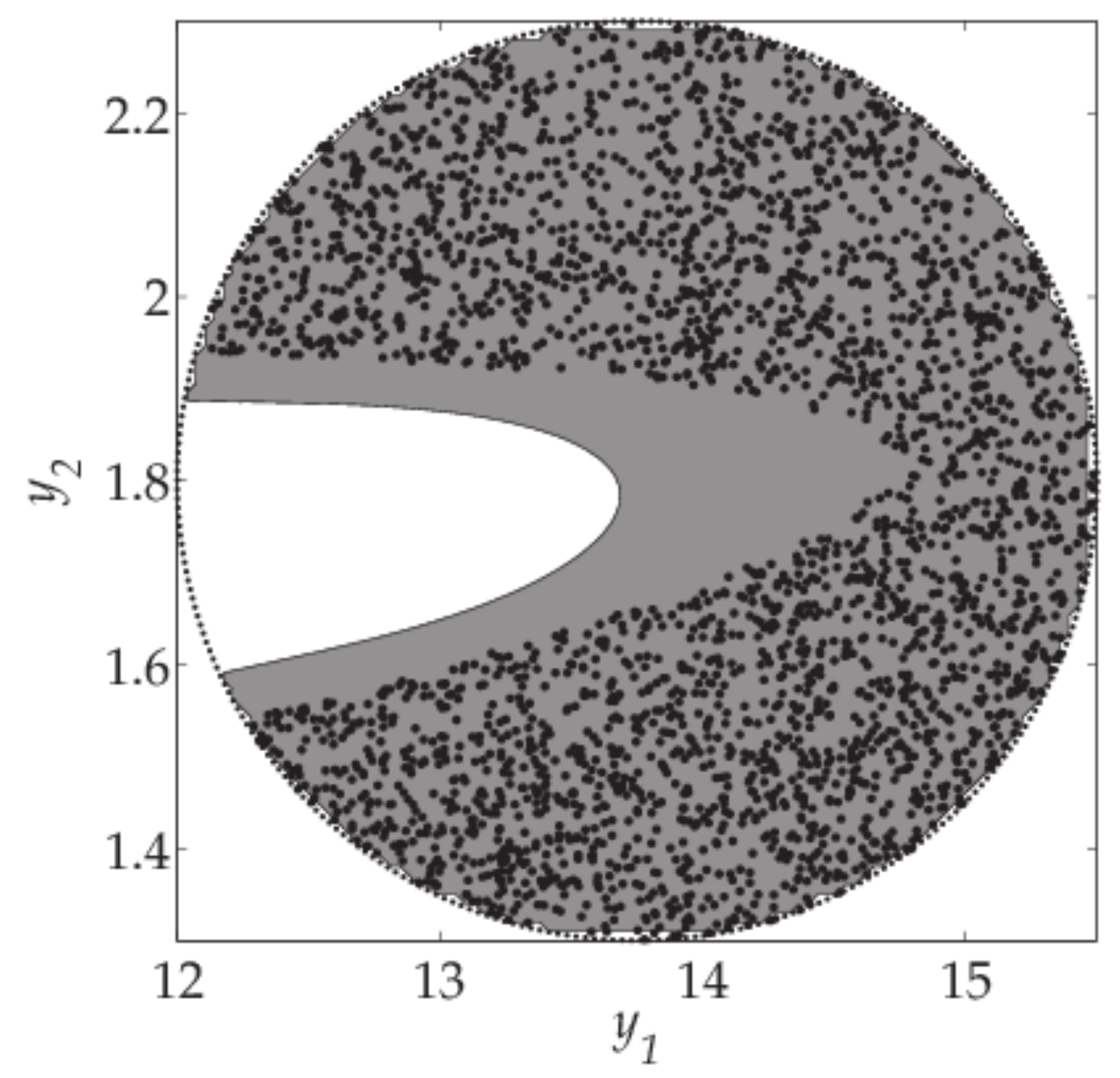}}
 \caption{Outer approximations $\F^1_r$ (light gray) of $\F \cap \B_0$ (black dot samples)
for Example~\ref{ex:pareto}, for $r=4,5$. 
 }	\label{fig:zoomexists}
\end{figure}

\subsection{semi-algebraic image of semi-algebraic sets}
\label{sec:saimage}
Given a semi-algebraic set $\S$ as in~\eqref{eq:defS}, Methods 1 and 2 can be extended to approximate the image of $\S$ under a semi-algebraic application $f = (f_1, \dots, f_m)$. To do so, we follow~\cite{LasPut10} and introduce lifting variables to represent non-polynomial components involved in $f_1, \dots, f_m$, as well as additional polynomial constraints.

Proceeding as in~\cite{LasPut10}, for each semi-algebraic function $f_j$, one introduces additional variables $\x^j := (x_1^j, \dots, x_{t_j}^j)$ such that the graph $\{(\x, f_j(\x)) : \x \in \S \} = \{(\x, x_{t_j}^j) : (\x, \x^j) \in \hat{\S}_j  \} $ for some semi-algebraic set $\hat{\S}_j \subseteq \R^{n + t_j}$. In the end, one works with the lifted set $\hat{\S} :=\{(\x, \x^1, \dots, \x^m) : (\x, \x^j) \in \hat{\S}_j \, , j = 1, \dots, m \,  \}$.

\begin{example}
\label{ex:saimage}
Here, we consider the image of the two-dimensional unit ball $\S := \{\x \in \R^2 : \| \x \|_2^2 \leq 1 \}$ under the semi-algebraic application $f(\x) := (\min (x_1+x_1 x_2, x_1^2), x_2-x_1^3) / 3$. 
Remind that $2 \min (a, b) = a + b - |a - b|$, so that $2 \min (x_1+x_1 x_2, x_1^2) = x_1+x_1 x_2 + x_1^2 - |x_1+x_1 x_2 - x_1^2|$.
To handle the absolute value, we introduce an additional variable $x_3$ together with the equality constraint $x_3^2 = (x_1 + x_1 x_2^2)^2$ and the inequality constraint $x_3 \geq 0$. 
\end{example}
\begin{figure}[!ht]
\centering
\subfigure[$r=1$]{
\includegraphics[scale=\sizetinyfig]{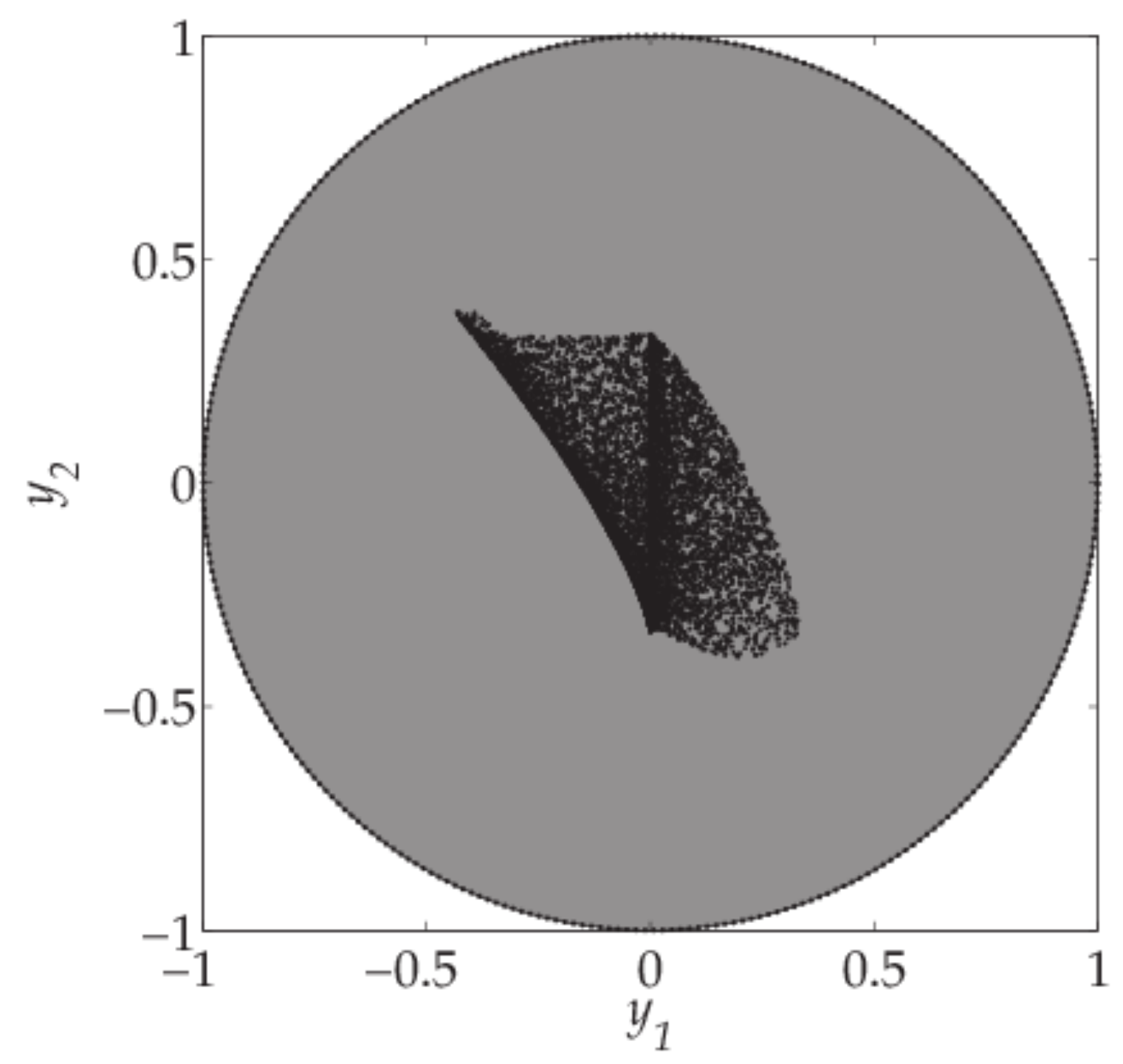}}
\subfigure[$r=2$]{
\includegraphics[scale=\sizetinyfig]{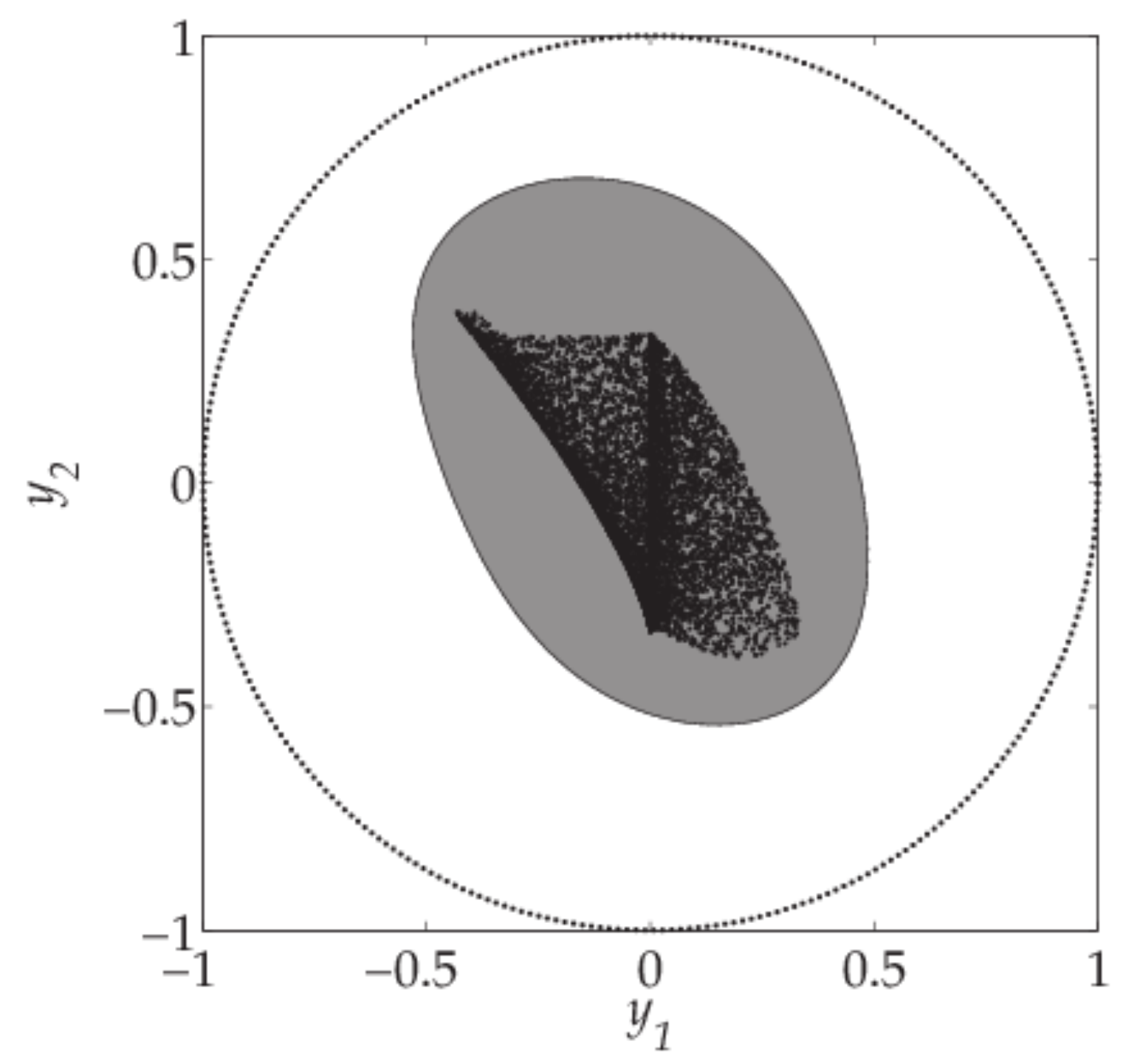}}
\subfigure[$r=3$]{
\includegraphics[scale=\sizetinyfig]{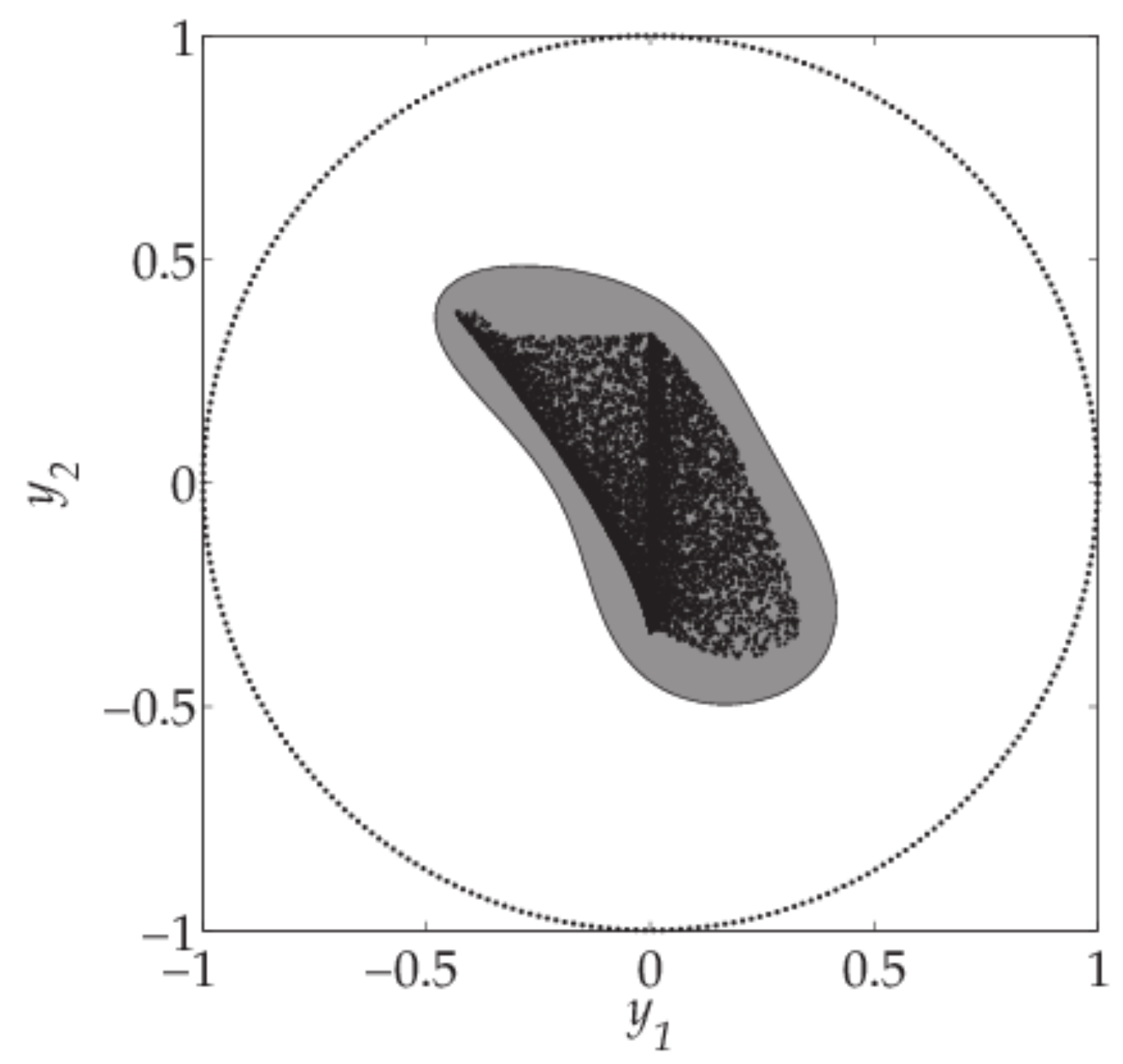}}
\subfigure[$r=4$]{
\includegraphics[scale=\sizetinyfig]{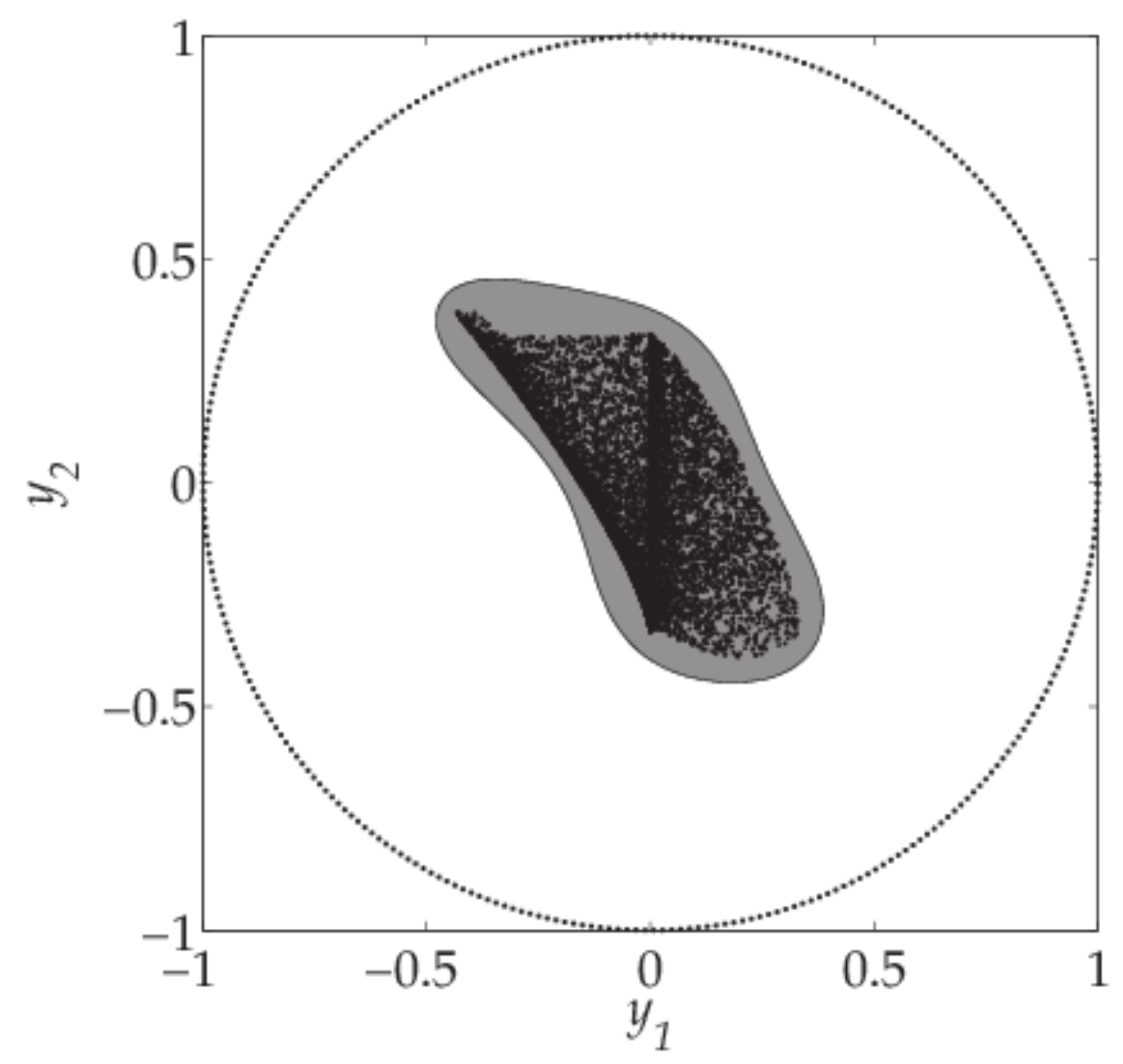}}
 \caption{Outer approximations $\F^1_r$ (light gray) of $\F$ (black dot samples)
for Example~\ref{ex:saimage}, for $r=1,2,3,4$. }	\label{fig:saimageexists}
\end{figure}
\begin{figure}[!ht]
\centering
\subfigure[$r=1$]{
\includegraphics[scale=\sizetinyfig]{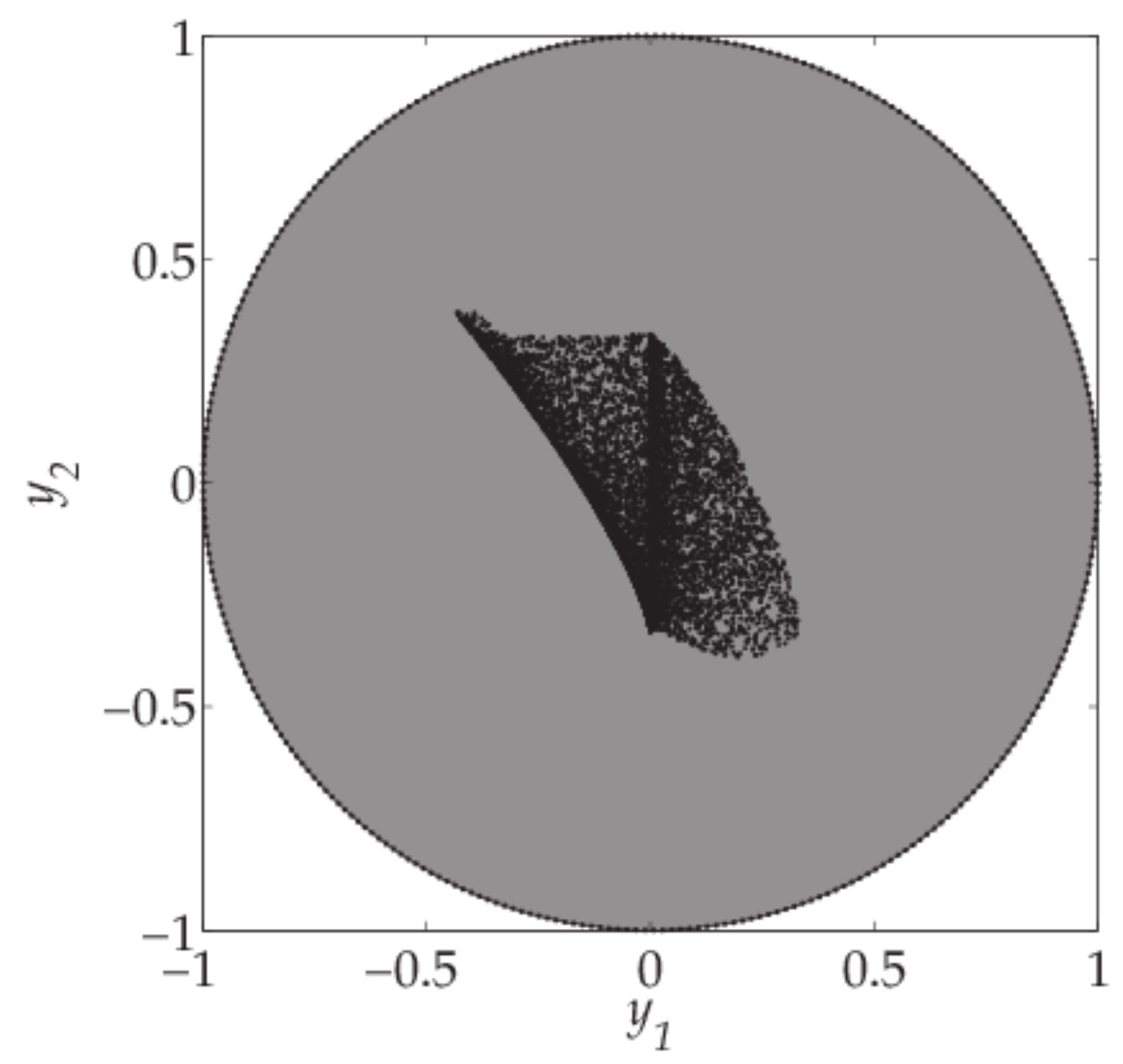}}
\subfigure[$r=2$]{
\includegraphics[scale=\sizetinyfig]{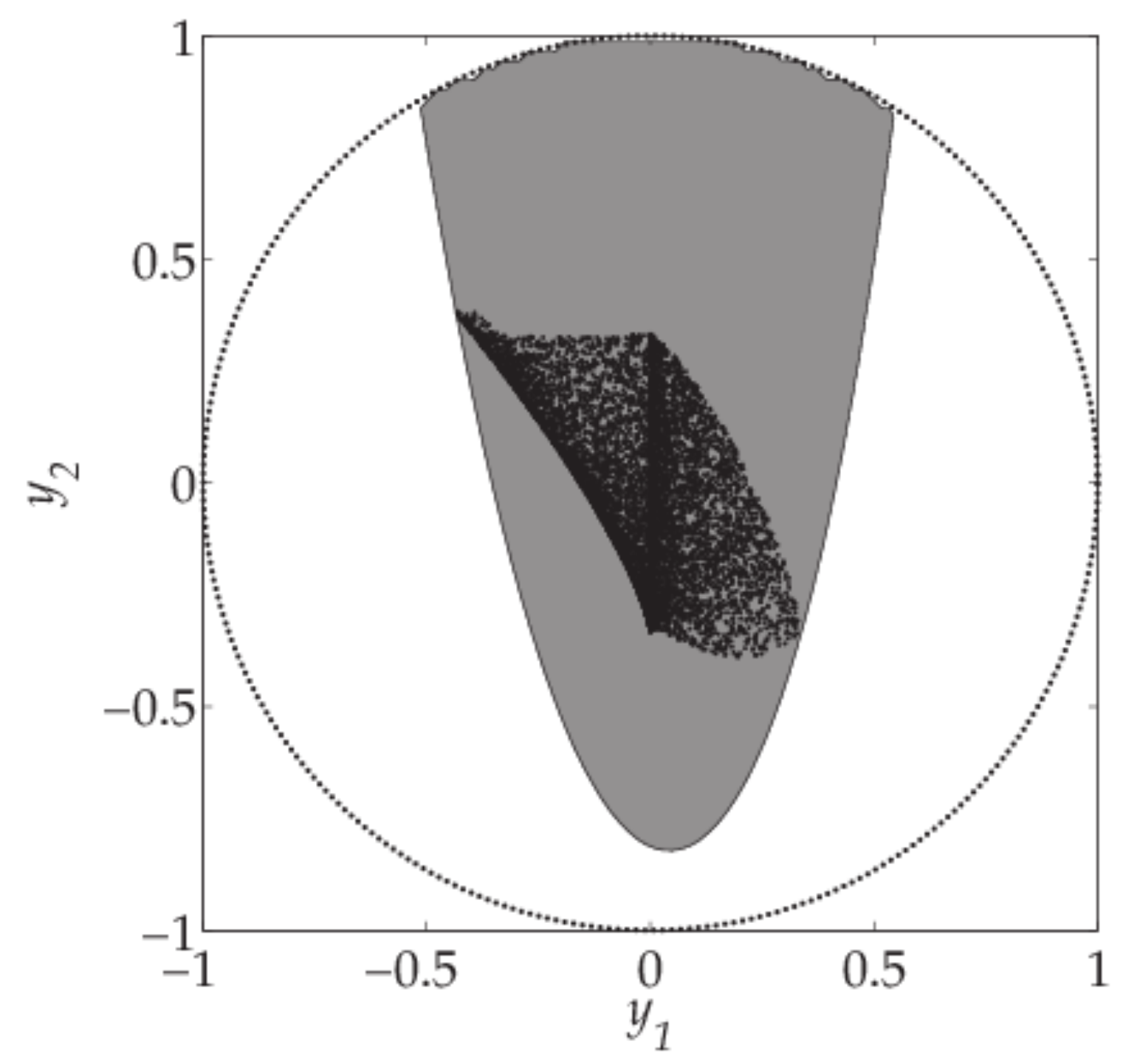}}
\subfigure[$r=3$]{
\includegraphics[scale=\sizetinyfig]{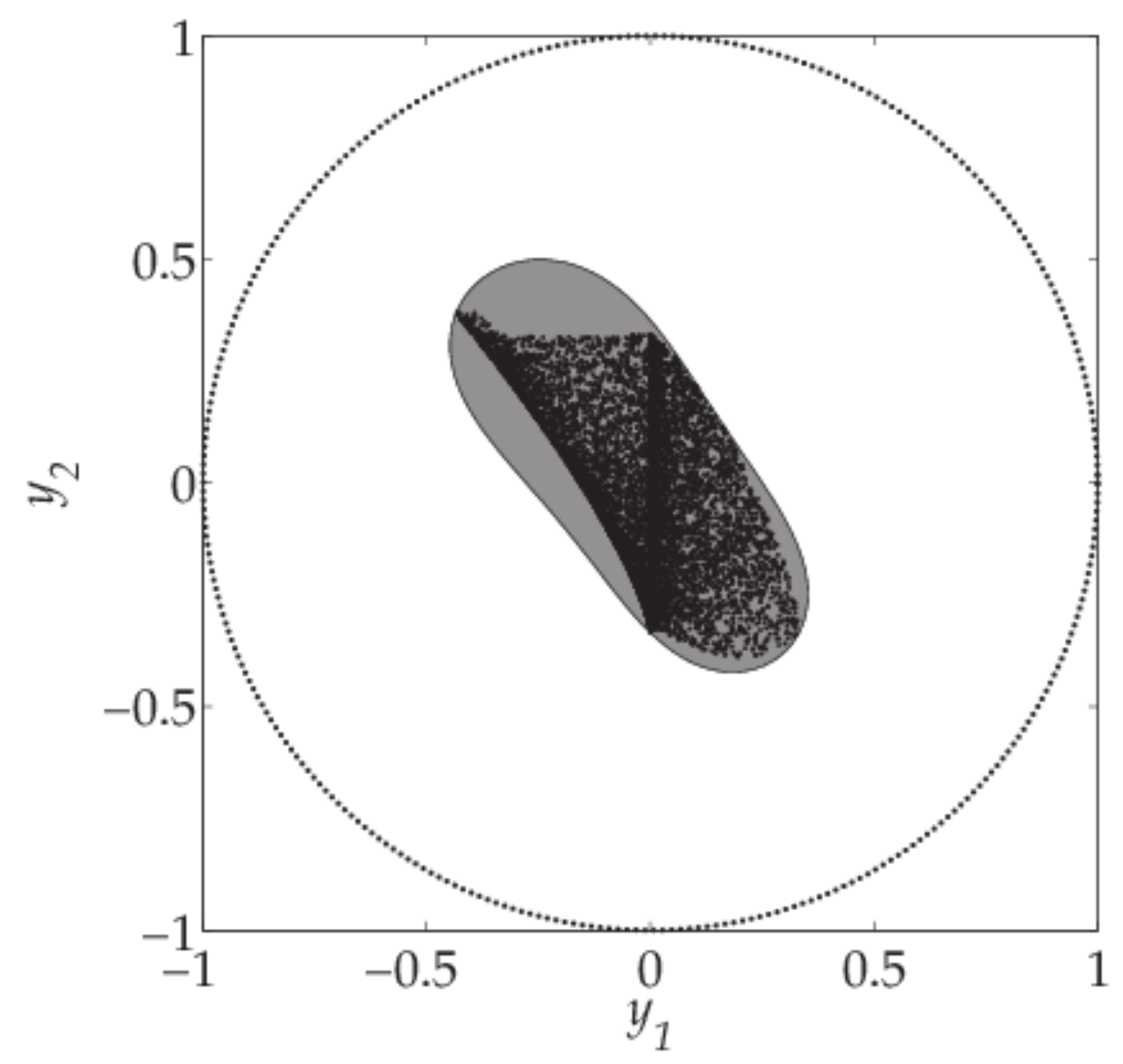}}
\subfigure[$r=4$]{
\includegraphics[scale=\sizetinyfig]{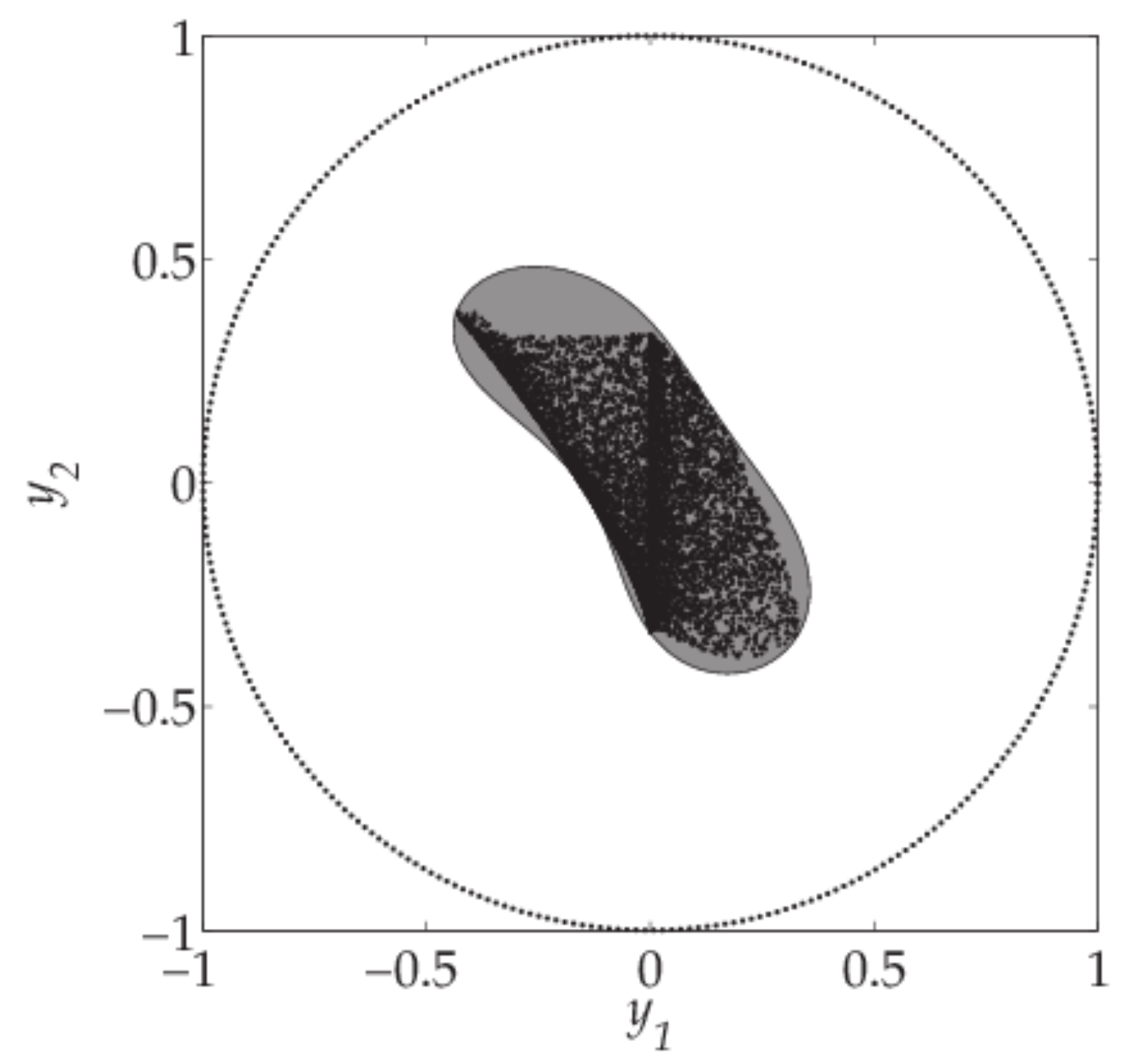}}
 \caption{Outer approximations $\F^2_r$ (light gray) of $\F$ (black dot samples)
for Example~\ref{ex:saimage}, for $r=1,2,3,4$.  }	\label{fig:saimagemeasureimage}
\end{figure}

As for Example~\ref{ex:ballimage}, we report in Table~\ref{table:saimage} the data related to the semidefinite problems solved by {\sc Mosek} to compute approximations of increasing degrees, while using Method 1, Method 2 and Method 2 with the lifting strategy. Method 2 fails to compute polynomial approximations of degree higher than eight ($r=4$), the system running out of memory (indicated with the symbol ``$-$''). The lifting strategy described in Section~\ref{sec:lift2} overcomes this practical limitation. 

\begin{table}[!ht]
\begin{center}
\caption{Comparison of timing results for Example~\ref{ex:saimage}}
\begin{tabular}{p{2.3cm}|c|ccccc}
\hline
\multicolumn{2}{c|}{relaxation order $r$}
 & 1& 2 & 3 & 4 & 5
\\
\hline            
\multirow{3}{*}{Method 1} & vars &  $66$ & $438$  & $3137$ & $16993$ &  $73213$ \\
& size & 45 & 226 & 1008 & 3387 & 9075\\
& time  (s) & $0.68 $ & $0.85 $ &  $1.16 $ & $14.41 $ & $ 147.32 $\\
\hline
\multirow{3}{*}{Method 2} &  vars & $715$  & $12243$ & $89695$  & $-$ & $-$\\
& size & 295 & 1957 & 6283 & $-$ & $-$ \\
 & time  (s) & $0.83 $ & $3.29 $ &  $52.55 $ & $-$ & $-$\\
\hline
\multirow{3}{2.3cm}{Method 2 with lifting} & vars &  $78$ & $540$ & $3788$ & $20216$  & $87475$ \\
& size & 51 & 273 & 1262 & 4247 & 11508 \\
& time (s) & $0.68 $ & $0.96 $ &  $1.83 $ & $14.44 $ & $174.80 $ \\
\hline
\end{tabular}
\label{table:saimage}
\end{center}
\end{table}
\section{Discussion and conclusion}
In this work, we propose two methods to approximate polynomial images of basic compact semi-algebraic sets, a numerical approximation alternative to exact computer algebra methods when the latter are too computationally demanding. 
In its present form, this methodology is applicable to problems of modest size, except if some sparsity can be taken into account, as explained earlier. 
Therefore, to handle larger size problems, the methodology needs to be adapted. 
A topic of further investigation is to search for alternative positivity certificates, less demanding than the SOS certificates used in this paper but more efficient than the LP based certificates as defined in~\cite{Handelman1988,Vasilescu}. On the one hand, the latter are appealing since they yield a hierarchy of LP relaxations (as opposed to semidefinite relaxations as in this paper). Moreover, today's LP solvers can handle huge size LP problems, which is far from being the case for semidefinite solvers. On the other hand, it has been shown in~\cite{lasserre2009moments} that generically finite convergence cannot occur for convex problems, except for the linear case. Finally, it could be interesting to look at various compactification procedures to study how the methodology could be generalized to non-compact situations.

\section*{Acknowledgments}
This work was partly funded by an award of the {\em Simone and Cino del Duca foundation} of Institut de France, a grant of the Gaspard Monge program for optimisation and operations research (PGMO), funded by the {\em Fondation Math\'ematiques Jacques Hadamard}. The authors would like to thank Mohab Safey El Din, Pierre-Lo\"ic Garoche, J\'erome Roussel, Alain Sarlette and Xavier Thirioux for fruitful discussions. The authors are also very grateful to the two reviewers for their careful reading, their criticism as well as their detailed feedback.
\bibliographystyle{alpha}

\end{document}